\documentclass[12pt,reqno]{amsart}
\usepackage{a4wide}
\usepackage{amsfonts,amsmath,amssymb,amscd}
\usepackage{amsthm}
\usepackage{amsmath,todonotes}
\usepackage{amscd}
\usepackage{cancel}
\usepackage{verbatim}
\usepackage{color}
\usepackage[all]{xy}
\usepackage[mathscr]{eucal}
\usepackage{mathrsfs}
\usepackage{fullpage}
\usepackage{enumitem}
\usepackage{hyperref}
\usepackage{bbm}
\usepackage{qtree}
\usepackage{mathrsfs}
\usepackage{bm}
\usepackage{bigints}
\usepackage{url}
\allowdisplaybreaks

\renewcommand{\a}{\alpha}
\renewcommand{\b}{\beta}
\newcommand{\e}{\epsilon}
\renewcommand{\d}{{\delta}}
\newcommand{\G}{\Gamma}
\renewcommand{\l}{\lambda}

\newcommand{\z}{\zeta}
\usepackage[headheight=110pt,top=0.6in, bottom=0.7in, left=0.6in, right=0.6in]{geometry}
\numberwithin{equation}{section}

\newtheorem{remark}[subsection]{Remark}

\newtheorem{theorem}{Theorem}

\newtheorem{lemma}[theorem]{Lemma}

\theoremstyle{definition}

\numberwithin{theorem}{section} 
\numberwithin{equation}{section}
\numberwithin{table}{section}

\DeclareMathOperator{\Res}{Res}

\newcommand{\g}{\gamma}

   \makeatletter
\def\proof{\@ifnextchar[{\@oproof}{\@nproof}}
\def\@oproof[#1][#2]{\trivlist\item[\hskip\labelsep\textit{#2 Proof of\
#1.}~]\ignorespaces}
\def\@nproof{\trivlist\item[\hskip\labelsep\textit{Proof.}~]\ignorespaces}

\begin{document}
	\title[Lambert series of logarithm and the derivative of Deninger's function $R(z)$]{Lambert series of logarithm, the derivative of Deninger's function $R(z)$ and a mean value theorem for $\zeta\left(\frac{1}{2}-it\right)\zeta'\left(\frac{1}{2}+it\right)$}
	\author{Soumyarup Banerjee, Atul Dixit and Shivajee Gupta}
	\address{Department of Mathematics, Indian Institute of Technology Gandhinagar, Palaj, Gandhinagar 382355, Gujarat, India}\email{soumyarup.b@iitgn.ac.in}
	\address{Department of Mathematics, Indian Institute of Technology Gandhinagar, Palaj, Gandhinagar 382355, Gujarat, India}\email{adixit@iitgn.ac.in}
		\address{Department of Mathematics, Indian Institute of Technology Gandhinagar, Palaj, Gandhinagar 382355, Gujarat, India}\email{shivajee.o@iitgn.ac.in}
	
	\medskip
	\begin{abstract}
		An explicit transformation for the series $\sum\limits_{n=1}^{\infty}\displaystyle\frac{\log(n)}{e^{ny}-1},$ Re$(y)>0$, which takes $y$ to $1/y$, is obtained for the first time. This series transforms into a series containing $\psi_1(z)$, the derivative of Deninger's function $R(z)$. In the course of obtaining the transformation, new important properties of $\psi_1(z)$ are derived, as is a new representation for the second derivative of the two-variable Mittag-Leffler function $E_{2, b}(z)$ evaluated at $b=1$. Our transformation readily gives the complete asymptotic expansion of $\sum\limits_{n=1}^{\infty}\displaystyle\frac{\log(n)}{e^{ny}-1}$ as $y\to0$. An application of the latter  is that it gives the asymptotic expansion of $ \displaystyle\int_{0}^{\infty}\zeta\left(\frac{1}{2}-it\right)\zeta'\left(\frac{1}{2}+it\right)e^{-\delta t}\, dt$ as $\delta\to0$.
		\end{abstract}
	
	\dedicatory{Dedicated to Christopher Deninger on account of his $64^{\textup{th}}$ birthday}
	\maketitle
	
	\section{Introduction}\label{intro}
	Eisenstein series are the building blocks of modular forms and thus lie at the heart of the theory. In the case of the full modular group $\textup{SL}_{2}\left(\mathbb{Z}\right)$, the Eisenstein series of even integral weight $k\geq2$ are given by
	 \begin{equation}\label{eisenstein_defn}
	 E_k(z):=1-\frac{2k}{B_k}\sum_{n=1}^{\infty}\frac{n^{k-1}q^n}{1-q^n},
	 \end{equation}
 where $q=e^{2\pi iz}$ with $z\in\mathbb{H}$ (the upper-half plane) and $B_k$ are the Bernoulli numbers. They satisfy the modular transformations
\begin{align}
E_k(z+1)&=E_k(z)\hspace{5mm}(k\geq 2),\nonumber\\
 E_{k}\left(\frac{-1}{z}\right)&=z^kE_{k}(z)\hspace{2mm}(k>2), \hspace{5mm}
 E_{2}\left(\frac{-1}{z}\right)=z^2E_{2}(z)+\frac{6z}{\pi i}\label{inversion}.
\end{align}
The series on the right-hand side of \eqref{eisenstein_defn} is an example of what is known as a \emph{Lambert series} whose general form is
\begin{equation}\label{equivalentlambert}
\sum_{n=1}^{\infty}\frac{a(n)q^{n}}{1-q^n}=\sum_{n=1}^{\infty}\frac{a(n)}{e^{ny}-1}=\sum_{n=1}^{\infty}(1*a)(n)e^{-ny},
\end{equation}
where $q=e^{-y}$ with Re$(y)>0$, and $a(n)$ is an arithmetic function with $(1*a)(n)=\sum_{d|n}a(d)$ as the Dirichlet convolution.

For Re$(s)>1$, the Riemann zeta function is defined by $\zeta(s)=\sum_{n=1}^{\infty}n^{-s}$. Ramanujan \cite[p.~173, Ch. 14, Entry 21(i)]{ramnote}, \cite[p.~319-320, formula (28)]{lnb}, \cite[p.~275-276]{bcbramsecnote} derived a beautiful transformation involving the Lambert series associated with $a(n)=n^{-2m-1}, $ $m\in\mathbb{Z}\backslash\{0\}$,  and the odd zeta value $\zeta(2m+1)$, namely, for Re$(\a), \textup{Re}(\b)>0$ with $\a\b=\pi^2$,
	\begin{align}\label{zetaodd}
		\a^{-m}\left\{\frac{1}{2}\zeta(2m+1)+\sum_{n=1}^{\infty}\frac{n^{-2m-1}}{e^{2\a n}-1}\right\}&=(-\b)^{-m}\left\{\frac{1}{2}\zeta(2m+1)+\sum_{n=1}^{\infty}\frac{n^{-2m-1}}{e^{2\b n}-1}\right\}\nonumber\\
		&\quad-2^{2m}\sum_{j=0}^{m+1}\frac{(-1)^jB_{2j}B_{2m+2-2j}}{(2j)!(2m+2-2j)!}\a^{m+1-j}\b^j.
	\end{align}
	Along with the transformations of the Eisenstein series in \eqref{inversion}, this formula also encapsulates the transformations of the corresponding Eichler integrals of the Eisenstein series as well as the transformation property of the Dedekind eta-function. The literature on this topic is vast with many generalizations and analogues for other $L$-functions, for example, \cite{banerjeekumar}, \cite{berndtacta1975}, \cite{berndtrocky}, \cite{bradley2002}, \cite{kzf}, \cite{hhf1}, \cite{DGKM}, \cite{dkk},  \cite{dixitmaji1}, \cite{guptamaji}, \cite{katayama}, and \cite{komori}. See also the recent survey article \cite{berndtstraubzeta}.

Recently, Kesarwani, Kumar and the second author \cite[Theorem 2.4]{dkk} obtained a new generalization of \eqref{zetaodd}, namely, for $\textup{Re}(y)>0$ and any complex $a$ such that $\textup{Re}(a)>-1$,
	\begin{align}\label{maineqn}
		&\sum_{n=1}^\infty  \sigma_a(n)e^{-ny}+\frac{1}{2}\left(\left(\frac{2\pi}{y}\right)^{1+a}\mathrm{cosec}\left(\frac{\pi a}{2}\right)+1\right)\zeta(-a)-\frac{1}{y}\zeta(1-a)\nonumber\\
		&=\frac{2\pi}{y\sin\left(\frac{\pi a}{2}\right)}\sum_{n=1}^\infty \sigma_{a}(n)\Bigg(\frac{(2\pi n)^{-a}}{\Gamma(1-a)} {}_1F_2\left(1;\frac{1-a}{2},1-\frac{a}{2};\frac{4\pi^4n^2}{y^2} \right) -\left(\frac{2\pi}{y}\right)^{a}\cosh\left(\frac{4\pi^2n}{y}\right)\Bigg),
	\end{align}
where ${}_1F_{2}(a;b, c;z)$ is the generalized hypergeometric function
\begin{equation*}
	{}_1F_2(a;b, c; z):=\sum_{n=0}^{\infty}\frac{(a)_n}{(b)_n(c)_n}\frac{z^n}{n!},
\end{equation*}
where $z\in\mathbb{C}$ and $(a)_n=a(a+1)\cdots(a+n-1)$. They \cite[Theorem 2.5]{dkk} also analytically continued this result to Re$(a)>-2m-3$, $m\in\mathbb{N}\cup\{0\}$, and in this way, they were able to get as corollaries not only Ramanujan's formula \eqref{zetaodd} and the transformation formula for the Dedekind eta-function but also new transformations when $a$ is an even integer. Moreover, they showed \cite[Equation (2.19)]{dkk} that letting $a\to0$ in \eqref{maineqn} gives the following transformation of Wigert \cite[p.~203, Equation (A)]{wig0}:
\begin{align}\label{kanot}
	\sum_{n=1}^{\infty}\frac{1}{e^{ny}-1}=\frac{1}{4}+\frac{\gamma-\log(y)}{y}+\frac{2}{y}\sum_{n=1}^{\infty}\left\{\log\left(\frac{2\pi n}{y}\right)-\frac{1}{2}\left(\psi\left(\frac{2\pi in}{y}\right)+\psi\left(-\frac{2\pi in}{y}\right)\right)\right\},
\end{align}
where $\psi(z):=\Gamma'(z)/\Gamma(z)$ is the logarithmic derivative of the gamma function commonly known as the digamma function. Wigert \cite[p.~203]{wig0} called this transformation \emph{`la formule importante'} (an important formula). Indeed, it is important, for, if we let $y\to 0$ in any angle $|\arg(y)|\leq\lambda$, where $\lambda<\pi/2$, it gives the complete asymptotic expansion upon using \eqref{asypsi}, that is,
\begin{align}\label{wiganobef}
	\sum_{n=1}^{\infty} d(n)e^{-ny}\sim\frac{1}{4}+\frac{(\gamma-\log (y))}{y}-\sum_{n=1}^{\infty}\frac{B_{2n}^2y^{2n-1}}{(2n)(2n)!},
\end{align}
where $\gamma$ is Euler's constant, which, in turn, readily implies \cite[p.~163, Theorem 7.15]{titch}
\begin{align}\label{wigano}
\sum_{n=1}^{\infty} d(n)e^{-ny}=\frac{1}{4}+\frac{\gamma-\log y}{y}-\sum_{n=0}^{N-1}\frac{B_{2n+2}^2}{(2n+2)(2n+2)!} y^{2n+1}+O(|y|^{2N}).
\end{align}

The study of the moments
\begin{equation}\label{moment1}
M_{k}(T):=\int_{0}^{T}\left|\zeta\left(\frac{1}{2}+it\right)\right|^{2k}\, dt
\end{equation}
 is of fundamental importance in the theory of the Riemann zeta function. It is conjectured that $M_{k}(T)\sim C_k T\log^{k^2}(T)$ as $T\to\infty$ for positive constants $C_k$ although for $k=1$ and $2$, this has been proved by Hardy and Littlewood \cite{hl} and Ingham \cite{ingham} respectively. Such results are known as mean-value theorems for the zeta function. The importance of the study of moments lies, for example, in the fact that the estimate $M_k(T)=O_{k,\epsilon}(T^{1+\epsilon})$ for every natural number $k$ is equivalent to the Lindel\"{o}f hypothesis \cite{soundararajan} (see also \cite{harlit}).

 Another set of mean-value theorems which plays an important role in the theory is the one concerning the asymptotic behavior of the smoothly weighted moments, namely,
\begin{equation}\label{moment2}
\int_{0}^{\infty}\left|\zeta\left(\frac{1}{2}+it\right)\right|^{2k}e^{-\delta t}\, dt
\end{equation}
as $\delta\to0$. The relation between the two types of moments in \eqref{moment1} and \eqref{moment2} is given by a result  \cite[p.~159]{titch} which states that if $f(t)\geq0$ for all $t$ and for a given positive $m$,
\begin{equation*}
	\int_{0}^{\infty}f(t)e^{-\delta t}\, dt\sim\frac{1}{\delta}\log^{m}\left(\frac{1}{\delta}\right)
\end{equation*}
as $\delta\to0$, then 
\begin{equation*}
	\int_{0}^{T}f(t)\, dt\sim T\log^{m}(T)
\end{equation*}
as $T\to\infty$. For an excellent survey on the moments, we refer the reader to \cite{ivic}. 

The asymptotic expansion of $\sum_{n=1}^{\infty} d(n)e^{-ny}$ as $y\to0$ in \eqref{wigano} allows us to obtain the asymptotic estimate for the smoothly weighted second moment, namely, as $\delta\to0$, for every natural number $N$, we have \cite[p.~164, Theorem 7.15 (A)]{titch}
\begin{align}\label{sw2nd}
	\int_{0}^{\infty}\left|\zeta\left(\frac{1}{2}+it\right)\right|^{2}e^{-\delta t}\, dt=\frac{\gamma-\log(2\pi\delta)}{2\sin(\delta/2)}+\sum_{n=0}^{N}c_n\delta^n+O(\delta^{N+1}),
\end{align}
where the  $c_n$ are constants and the constant implied by the big-O depends on $N$. A simple proof of \eqref{sw2nd} was given by Atkinson \cite{atkinson}. In fact, \eqref{wiganobef} gives the complete asymptotic expansion for this moment.

The primary goal of this paper is to give a non-trivial application of \eqref{maineqn}. Note that the complex variable $a$ in \eqref{maineqn} enables differentiation of \eqref{maineqn} with respect to $a$, which is not possible in \eqref{zetaodd} or other such known results. Indeed, it is an easy affair to check that differentiating the series on the left-hand side of \eqref{maineqn} with respect to $a$ and then letting $a=0$ gives  $\sum\limits_{n=1}^{\infty}\displaystyle\frac{\log(n)}{e^{ny}-1}$, which, in view of \eqref{equivalentlambert}, satisfies
\begin{align*}
	\sum_{n=1}^{\infty}\frac{\log(n) }{e^{ny}-1}
	=\sum_{n=1}^{\infty}\log\big(\prod_{d|n}d\big)e^{-ny}
	=\frac{1}{2}\sum_{n=1}^{\infty}d(n)\log(n)e^{-ny}.
\end{align*}
Here, in the last step, we used an elementary result $\prod_{d|n}d=n^{d(n)/2}$; see, for example, \cite[Exercise 10, p.~47]{Apostol}.

What is surprising though is, differentiating the right-hand side of \eqref{maineqn} with respect to $a$ and then letting $a=0$ leads to an explicit and interesting  series involving a well-known special function.  This special function deserves a separate mention and hence after its brief introduction here, the literature on it is discussed in detail in Section \ref{derivative-deninger}.

 In a beautiful paper \cite{deninger}, Deninger comprehensively studied the function $R:\mathbb{R}^{+}\to\mathbb{R}$ uniquely defined by the difference equation 
\begin{equation*}
	R(x+1)-R(x)=\log^{2}(x)\hspace{8mm}(R(1)=-\zeta''(0)).
\end{equation*}
The function $R(x)$ is an analogue of $\log(\Gamma(x))$ in view of the fact that the latter satisfies the difference equation $f(x+1)-f(x)=\log(x)$ with the initial condition $f(1)=0$. As noted in \cite[Remark (2.4)]{deninger}, $R$ can be analytically continued to
\begin{equation*} 
	\mathbb{D}:=\mathbb{C}\backslash\{x\in\mathbb{R}|\  x\leq0\}.
\end{equation*}
 The special function which appears in our main theorem, that is, in Theorem \ref{loglamb}, is $\psi_1(z)$, which is essentially the derivative of Deninger's function $R(z)$ (see \eqref{psi1r} below). For $z\in \mathbb{D}$, it is given by
\begin{equation}\label{psi1z}
	\psi_1(z)=-\gamma_1-\frac{\log(z)}{z}-\sum_{n=1}^{\infty}\left(\frac{\log(n+z)}{n+z}-\frac{\log(n)}{n}\right),
\end{equation}
where $\gamma_1$ is the first Stieltjes constant.

We are now ready to state the main result of our paper which transforms the Lambert series of logarithm into an infinite series consisting of $\psi_1(z)$.
 	\begin{theorem}\label{loglamb}
 	Let $\psi_1(z)$ be given in \eqref{psi1z}.
 	Then, for $\textup{Re}(y)>0$,
 	\begin{align}\label{translog}
 		\sum_{n=1}^{\infty}\frac{\log(n) }{e^{ny}-1}&=-\frac{1}{4}\log(2\pi)+\frac{1}{2y}\log^{2}(y)-\frac{\gamma^2}{2y}+\frac{\pi^2}{12y}\nonumber\\
 		&\quad-\frac{2}{y}(\gamma+\log(y))\sum_{n=1}^{\infty}\left\{\log\left(\frac{2\pi n}{y}\right)-\frac{1}{2}\left(\psi\left(\frac{2\pi in}{y}\right)+\psi\left(-\frac{2\pi in}{y}\right)\right)\right\}\nonumber\\
 		&\quad+\frac{1}{y}\sum_{n=1}^{\infty}\left\{\psi_1\left(\frac{2\pi in}{y}\right)+\psi_1\left(-\frac{2\pi in}{y}\right)-\frac{1}{2}\left(\log^{2}\left(\frac{2\pi in}{y}\right)+\log^{2}\left(-\frac{2\pi in}{y}\right)\right)+\frac{y}{4n}\right\}.
 	\end{align}
 	Equivalently,
 	\begin{align}\label{translog1}
 		&y\sum_{n=1}^{\infty}\frac{\gamma+\log(ny) }{e^{ny}-1}-\frac{1}{4}y\log(y)+y\left(\frac{1}{4}\log(2\pi)-\frac{\gamma}{4}\right)+\frac{1}{2}\log^2 (y)-\frac{\gamma^2}{2}-\frac{\pi^2}{12}\nonumber\\
 		&=\sum_{n=1}^{\infty}\left\{\psi_1\left(\frac{2\pi in}{y}\right)+\psi_1\left(-\frac{2\pi in}{y}\right)-\frac{1}{2}\left(\log^{2}\left(\frac{2\pi in}{y}\right)+\log^{2}\left(-\frac{2\pi in}{y}\right)\right)+\frac{y}{4n}\right\}.
 	\end{align}
 \end{theorem}
 \begin{remark}
 That the series
 \begin{equation*}
\sum_{n=1}^{\infty}\left\{\psi_1\left(\frac{2\pi in}{y}\right)+\psi_1\left(-\frac{2\pi in}{y}\right)-\frac{1}{2}\left(\log^{2}\left(\frac{2\pi in}{y}\right)+\log^{2}\left(-\frac{2\pi in}{y}\right)\right)+\frac{y}{4n}\right\} 
 \end{equation*}
 converges absolutely is clear from \eqref{asyseries} below.
 \end{remark}
The exact transformation in \eqref{translog}  is an analogue of Wigert's result \eqref{kanot}. This is evident from the fact that 
\begin{align*}
\log\left(\tfrac{2\pi n}{y}\right)-\tfrac{1}{2}\left(\psi\left(\tfrac{2\pi in}{y}\right)+\psi\left(-\tfrac{2\pi in}{y}\right)\right)=-\tfrac{1}{2}\left\{\psi\left(\tfrac{2\pi in}{y}\right)+\psi\left(-\tfrac{2\pi in}{y}\right)-\left(\log\left(\tfrac{2\pi i n}{y}\right)+\log\left(\tfrac{-2\pi i n}{y}\right)\right)\right\},
\end{align*}
which should be compared with the summand of the second series on the right-hand side of \eqref{translog}.
 
 Thus, \eqref{translog} allows us to transform $\sum_{n=1}^{\infty}d(n)\log(n)e^{-ny}$ into series having a constant times $1/y$ in the arguments of the functions in their summands. This ``modular'' behavior has an instant application: it gives the complete asymptotic expansion of $\sum_{n=1}^{\infty}d(n)\log(n)e^{-ny}$ as $y\to0$, which is given in the following result.
 
\begin{theorem}\label{logy0}
	Let $A$ denote the Glaisher-Kinkelin constant defined by \cite{glaisher1}, \cite{glaisher2}, \cite{kinkelin}, \cite[p.~461, Equation (A.7)]{voros}
	\begin{align*}
		\log(A):=\lim_{n\to\infty}\left\{\sum_{k=1}^{n}k\log(k)-\left(\frac{n^2}{2}+\frac{n}{2}+\frac{1}{12}\right)\log(n)+\frac{n^2}{4}\right\}.
	\end{align*}
	As $y\to0$ in $|\arg(y)|<\pi/2$,
		\begin{align}\label{logy0eqn}
	\sum_{n=1}^{\infty}\frac{\log(n) }{e^{ny}-1}&\sim\frac{1}{2y}\log^{2}(y) + \frac{1}{y}\left(\frac{\pi^2}{12} - \frac{\gamma^2}{2} \right)-\frac{1}{4}\log(2\pi) +\frac{y}{12}\left(\log A - \frac{1}{12} \right) \nonumber\\
&\quad +\sum_{k=2}^{\infty} \frac{B_{2k} y^{2k-1}}{k}\left\{\frac{B_{2k}}{2(2k)!}\left(\gamma - \sum_{j=1}^{2k-1}\frac{1}{j}+\log (2\pi)\right)+ \frac{(-1)^k \zeta'(2k)}{(2\pi)^{2k}}\right\}.
	\end{align}
\end{theorem} 
 As seen earlier, Wigert's result \eqref{wigano} is useful in getting the asymptotic estimate for the smoothly weighted second moment of the zeta function on the critical line given in \eqref{sw2nd}. It is now natural to ask whether our result in Theorem \ref{logy0} has an application in the theory of moments. Indeed, \eqref{logy0eqn} implies the following result.
 \begin{theorem}\label{moments}
 As $\delta\to0, |\arg(\delta)|<\pi/2$,
 \begin{align}\label{mvtder}
 \int_{0}^{\infty}\zeta\left(\frac{1}{2}-it\right)\zeta'\left(\frac{1}{2}+it\right)e^{-\delta t}\, dt=\frac{-\log^{2}(2\pi\delta)+\gamma^2-\frac{\pi^2}{6}}{4\sin\left(\frac{\delta}{2}\right)}+\sum_{k=0}^{2m-2}d_k\delta^k+O\left(\delta^{2m-1}\right),
 \end{align}
 where the $d_k$ are effectively computable constants and the constant implied by the big-$O$ depends on $m$. 
 \end{theorem}
In fact, one can obtain the complete asymptotic expansion of the left-hand side of \eqref{mvtder} using \eqref{logy0eqn}.

Mean value theorems involving the derivatives of the Riemann zeta function have been well-studied. For example, Ingham \cite{ingham} (see also Gonek \cite[Equation (2)]{gonek}\footnote{There is a slight typo in the asymptotic formula on the right-hand side of this equation in that $(-1)^{\mu+\nu}$ is missing.}) showed that
\begin{equation*}
\int_{0}^{T}\zeta^{(\mu)}\left(\frac{1}{2}+it\right)\zeta^{(\nu)}\left(\frac{1}{2}-it\right)\, dt\sim\frac{(-1)^{\mu+\nu}\hspace{1mm}T}{\mu+\nu+1}\log^{\mu+\nu+1}(T)
\end{equation*}
as $T\to\infty$, where $\mu, \nu\in\mathbb{N}\cup\{0\}$. See also \cite[p.~102]{butterbaugh}. In particular, for $\mu=1$ and $\nu=0$, we have
\begin{align*}
\int_{0}^{T}\zeta\left(\frac{1}{2}-it\right)\zeta'\left(\frac{1}{2}+it\right)\, dt\sim \frac{-T}{2}\log^{2}(T)
\end{align*}
as $T\to\infty$.

The proof of Theorem \ref{loglamb} is quite involved in the sense that we had to establish several new results in the course of proving it. These results are important in themselves and may have applications in other areas. Hence this paper is organized as follows. 

We collect frequently used results in the next section. In \S 3, we first prove Theorem \ref{asymptotic-psi1} which gives an asymptotic expansion of $\psi_1(z)$ followed by Theorem \ref{kloosterman-type}, a Kloosterman-type result for $\psi_1(z)$. Theorem \ref{analoguedgkm} is the highlight of this section and is an analogue of \eqref{dgkmresult} established in \cite[Theorem 2.2]{DGKM}. \S 4 is devoted to obtaining a new representation for the second derivative of the two-variable Mittag-Leffler function $E_{2, b}(z)$ at $b=1$. In \S 5, we prove our main results, that is, Theorems \ref{loglamb}, \ref{logy0} and \ref{moments}. Finally, we conclude the paper with some remarks and directions for future research.

\section{Preliminaries}
Stirling's formula in a vertical strip $\alpha\leq\sigma\leq\beta$, $s=\sigma+it$ states that \cite[p.~224]{cop}
\begin{equation}\label{strivert}
|\Gamma(s)|=(2\pi)^{\tfrac{1}{2}}|t|^{\sigma-\tfrac{1}{2}}e^{-\tfrac{1}{2}\pi |t|}\left(1+O\left(\frac{1}{|t|}\right)\right)
\end{equation}
uniformly as $|t|\to\infty$.

We will also need the following result established in \cite[Theorem 2.2]{DGKM} which is valid for $\textup{Re}(w)>0$:
\begin{align}\label{dgkmresult}
	\sum_{n=1}^\infty\int_0^\infty\frac{t\cos(t)}{t^2+n^2w^2}\ dt=\frac{1}{2}\left\{\log\left(\frac{w}{2\pi}\right)-\frac{1}{2}\left(\psi\left(\frac{iw}{2\pi}\right)+\psi\left(-\frac{iw}{2\pi}\right)\right)\right\}.
\end{align}
Watson's lemma is a very useful result in the asymptotic theory of Laplace integrals $\displaystyle\int_{0}^{\infty}e^{-zt}f(t)\, dt$. This result typically holds for $|\arg(z)|<\pi/2$. However, with additional restrictions on $f$, Watson's lemma is known to hold for extended sectors. For the sake of completeness, we include it here in the form given in \cite[p.~14, Theorem 2.2]{temme2015}.
\begin{theorem}\label{watsextended}
Let $f$ be analytic inside a sector $D: \alpha<\arg(t)<\beta$, where $\alpha<0$ and $\beta>0$. For each $\delta\in\left(0, \frac{1}{2}\beta-\frac{1}{2}\alpha\right)$, as $t\to0$ in the sector $D_{\delta}: \alpha+\delta<\arg(t)<\beta-\delta$, we have
\begin{align*}
f(t)\sim t^{\lambda-1}\sum_{n=0}^{\infty}a_nt^n,
\end{align*}
where $\textup{Re}(\lambda)>0$. Suppose there exists a real number $\sigma$ such that $f(t)=O\left(e^{\sigma|t|}\right)$ as $t\to\infty$ in $D_{\delta}$. Then the integral 
\begin{equation}\label{flt}
F_{\lambda}(t):=\int_{0}^{\infty}e^{-zt}f(t)\, dt
\end{equation}
or its analytic continuation, has the asymptotic expansion
\begin{align*}
F_{\lambda}(z)\sim\sum_{n=0}^{\infty}a_n\frac{\G(n+\lambda)}{z^{n+\lambda}}
\end{align*}
as $z\to\infty$ in the sector
\begin{equation*}
-\beta-\frac{\pi}{2}+\delta<\arg(z)<-\alpha+\frac{\pi}{2}-\delta.
\end{equation*}
The many-valued functions $t^{\lambda-1}$ and $z^{n+\lambda}$ have their principal values on the positive real axis and are defined by continuity elsewhere.
\end{theorem}
Actually we will be using an analogue of the above theorem where the integrand in \eqref{flt} has a logarithmic factor.

We will be frequently using Parseval's theorem given next. Let $\mathfrak{F}(s)=M[f;s]$ and $\mathfrak{G}(s)=M[g;s]$ denote the Mellin transforms of functions $f(x)$ and $g(x)$ respectively and let $c=\textup{Re}(s)$. If $M[f;1-c-it]\in L(-\infty, \infty)$ and $x^{c-1}g(x)\in L[0,\infty)$, then Parseval's formula \cite[p.~83]{kp} is given by
\begin{equation}
	\int_{0}^{\infty} f(x) g(x) dx = \frac{1}{2 \pi i } 
	\int_{(c)}\mathfrak{F}(1-s)\mathfrak{G}(s)\, ds,
	\label{parseval-1}
\end{equation}
where the vertical line Re$(s) = c$ lies in the common strip of 
analyticity of the Mellin transforms $\mathfrak{F}(1-s)$ and $\mathfrak{G}(s)$, and, here and throughout the sequel, we employ the notation $\int_{(c)}$ to denote the line integral $\int_{c-i\infty}^{c+i\infty}$.

\section{New results on $\psi_1(z)$}\label{derivative-deninger}

In \cite{dilcher}, Dilcher studied in detail the generalized gamma function $\G_k(z)$ which relates to the Stieltjes constant $\gamma_k, k\geq0$, defined by\footnote{Note that Deninger's definition of $\gamma_1$ in \cite[p.~174]{deninger} involves an extra factor of $2$ which is not present in conventional definition of $\g_1$, that is, in the $k=1$ case of \eqref{sc}.}
\begin{equation}\label{sc}
	\gamma_k:=\lim_{n\to\infty}\left(\sum_{j=1}^{n}\frac{\log^{k}(j)}{j}-\frac{\log^{k+1}(n)}{k+1}\right),
\end{equation}
in a similar way as the Euler gamma function $\G(z)$ relates to the Euler constant $\g=\gamma_0$. Using \cite[Equation (2.1)]{dilcher} and \cite[Equation (2.3.1)]{deninger}, we see that Dilcher's $\Gamma_1(z)$ is related to Deninger's $R(z)$ by\footnote{By analytic continuation, Equation (2.3.1) from \cite{deninger} is valid for $z\in\mathbb{D}$.}
\begin{equation}\label{dilden}
	\log(\Gamma_1(z))=\frac{1}{2}(R(z)+\zeta''(0)).
\end{equation}
As mentioned by Deninger in \cite[Remark (2.4)]{deninger}, contrary to Euler's $\G$, the function $\exp(R(x))$, or equivalently $\G_1(x)$, where $x>0,$ cannot be meromorphically continued to the whole complex plane. But $\G_1(z)$ is analytic in $z\in\mathbb{D}$.
It is this $\exp(R(x))$ that Languasco and Righi \cite{languasco} call as the \emph{Ramanujan-Deninger gamma function}. They have also given a fast algorithm to compute it.

Dilcher also defined the generalized digamma function $\psi_k(z)$ as the logarithmic derivative of $\Gamma_k(z)$. His Proposition 10 from \cite{dilcher} implies that for $z\in\mathbb{D}$,
\begin{equation}\label{psikz}
	\psi_k(z)=-\gamma_k-\frac{\log^{k}(z)}{z}-\sum_{n=1}^{\infty}\left(\frac{\log^{k}(n+z)}{n+z}-\frac{\log^{k}(n)}{n}\right),
\end{equation}
where $k\in\mathbb{N}\cup\{0\}$. Its special case $k=1$ has already been given in \eqref{psi1z}. The function $\psi_k(z)$ occurs in Entry 22 of Chapter 8 in Ramanujan's second notebook, see \cite{RN_1}. It is also used by Ishibashi \cite{ishibashi} to construct the $k^{\textup{th}}$ order Herglotz function which, in turn, plays an important role in his evaluation of the Laurent series coefficients of a zeta function associated to an indefinite quadratic form. As noted by Ishibashi and Kanemitsu \cite[p.~78]{ishikan},
\begin{equation*}
	\psi_k(z)=\frac{1}{k+1}R_{k+1}'(z),
\end{equation*}
where $R_k(z)$ is defined by
\begin{equation}\label{rkz}
	R_k(z)=(-1)^{k+1}\frac{\partial^{k}}{\partial s^k}\zeta(0, z),
\end{equation}
and 
$R_2(z)=R(z)$ of Deninger. Thus,
\begin{equation}\label{psi1r}
	\psi_1(z)=\frac{1}{2}R'(z),
\end{equation} 
which is also implied by \eqref{dilden}.
The function $\psi_k(z)$ is related to the Laurent series coefficients $\gamma_k(z)$ of the Hurwitz zeta function $\zeta(s, z)$. To see this, from \cite[Theorem 1]{berndthurwitzzeta}, note that if 
\begin{equation}\label{hurwitzlaurent}
\zeta(s, z)=\frac{1}{s-1}+\sum_{k=0}^{\infty}\frac{(-1)^k\gamma_k(z)}{k!}(s-1)^k,
\end{equation}
then\footnote{It is to be noted that Berndt includes the factor $\frac{(-1)^{k}}{k!}$ in the definition of $\gamma_k(z)$ and does not have it in the summand of \eqref{hurwitzlaurent}.}
\begin{equation}\label{scz}
	\gamma_k(z)=\lim_{n\to\infty}\left(\sum_{j=0}^{n}\frac{\log^{k}(j+z)}{j+z}-\frac{\log^{k+1}(n+z)}{k+1}\right)
	\end{equation}
	so that $\gamma_k(1)=\gamma_k$. 
Then from \eqref{sc}, \eqref{psikz}, \eqref{scz} and the fact \cite[Lemma 1]{dilcher} that
\begin{equation*}
	\lim_{n\to\infty}\left(\log^{k+1}(n+z)-\log^{k+1}(n)\right)=0 \hspace{6mm} (z\in\mathbb{D}),
\end{equation*}
	it is not difficult to see that
\begin{equation*}
\psi_k(z)=-\gamma_k(z),
\end{equation*}
which was also shown by Shirasaka \cite[p.~136]{shirasaka}.
Further properties and applications of $\psi_k(z)$ are derived in \cite{dilcher1}. 

We thus see that the literature on $R(z)$, that is, $R_2(z)$, and, in general, on $R_k(z)$, is growing fast. In the words of Ishibashi \cite[p.~61]{ishibashi}, \emph{``Deninger proved several analytic properties of $R_2(x)$  in order to familiarize and assimilate it as one of the most commonly used number-theoretic special functions,\dots'.} Languasco and Righi \cite{languasco} have also given a fast algorithm to compute $\psi_1(x), x>0$.

Our first result of this section gives the asymptotic expansion of $\psi_1(z)$ for $z\in\mathbb{D}$. To accomplish it, we require a generalization of Watson's lemma which allows for a logarithmic factor in the integrand. Such an expansion seems to have been first obtained by Jones \cite[p.~439]{jones} (see also \cite[Equations (4.14), (4.15)]{wong-wyman}). Though we will be using the same expansion, it is useful  to rigorously derive it as a special case of a more general result due to Wong and Wyman \cite[Theorem 4.1]{wong-wyman} given in the following theorem.
We note in passing that Riekstins \cite{riekstins} has also obtained asymptotic expansions of integrals involving logarithmic factors.
 \begin{theorem}\label{Watsongen}
	For $\gamma\in\mathbb{R}$, define a function 
	\begin{align*}
		F(z):=\int_{0}^{\infty e^{i\g}}f(t)e^{-zt}dt. 
	\end{align*}
	Assume that $F(z)$ exists for some $z=z_0$. If 
	\begin{enumerate}
		\item for each integer $N\in \mathbb{Z}$ 
		\begin{align*}
			f(t)=\sum_{n=0}^{N}a_nt^{\l_n-1}P_n(\log t)+o(t^{\l_N-1}(\log t)^{m(N)}),
		\end{align*}
		as $t\to 0$ along $\arg(t)=\g$.
		\item $P_n(\omega)$ is a polynomial of degree $m=m(n)$.
		\item $\{\l_n\}$ is a sequence of complex numbers, with $\Re (\l_{n+1})>\Re( \l_{n}), \Re (\l_{0})>0,$ for all $n$ such that $n$ and $n+1$ are in $\mathbb{Z}$.
		\item $\{a_n\}$ is a sequence of complex numbers. 
	\end{enumerate}
	Then as $z\to \infty $ in $S(\Delta)$
	\begin{align*}
		F(z)\sim \sum_{n\in\mathbb{Z}} a_nP_n(D_n)[\Gamma(\l_n)z^{-\l_n}]+o\left(z^{-\l_n}(\log z)^{m(n)}\right),
	\end{align*}
	where $S(\Delta): |\arg (ze^{i\g})|\leqq{\pi\over 2}-\Delta$, and $D_n$ is the operator $D_n:={d\over d\l_n}$. This result is uniform in the approach of $z\to \infty$ in $S(\Delta)$.
\end{theorem}

\begin{theorem}\label{asymptotic-psi1}
Let $|\arg(z)|<\pi$. Then as $z\to\infty$,
\begin{align}\label{asymptotic-psi1_eqn}
\psi_1(z)\sim{1\over 2}\log^2(z)-{1 \over 2z}\log z+\sum_{k=1}^{\infty}{B_{2k}\over 2kz^{2k}} \left(\sum_{j=1}^{2k-2}{1\over j}+{1\over 2k-1}-\log z \right).
\end{align}
\end{theorem}
\begin{proof}
We first prove the result for $|\arg(z)|<\pi/2$ and then extend it to $|\arg(z)|<\pi$. To that end, we begin with the analogue for $R(z)$ of Plana's integral for $\log(\G(z))$, namely, for $\textup{Re}(z)>0$, we have \cite[Equation (2.12)]{deninger}
\begin{align}\label{plana}
R(z)=-\zeta''(0)-2\int_{0}^{\infty}\left((z-1)e^{-t}+\frac{e^{-zt}-e^{-t}}{1-e^{-t}}\right)\frac{\g+\log(t)}{t}\, dt.
\end{align}
Differentiating \eqref{plana} under the integral sign with respect to $z$ and using \eqref{psi1r}, we see that
\begin{align}\label{psi1lang}
\psi_1(z)&=-\int_{0}^{\infty}\left(e^{-t}-\frac{te^{-zt}}{1-e^{-t}} \right) \frac{(\g+\log( t))}{t}dt \nonumber \\
	& =-\int_{0}^{\infty}\left(e^{-t}-e^{-zt}\right){(\g+\log (t))}{ dt \over t} -\int_{0}^{\infty}e^{-zt}\left({1\over t}-\frac{1}{1-e^{-t}} \right) {(\g+\log( t))}dt\nonumber \\
	&={1\over 2} \log ^2(z) -\int_{0}^{\infty}e^{-zt}\left({1\over t}-\frac{1}{1-e^{-t}} \right) {(\g+\log( t))}dt,
	\end{align}
	where we used the fact that for Re$(z)>0$,
	\begin{align*}
	\int_{0}^{\infty}(e^{-t}-e^{-zt})(\g+\log(t))\frac{dt}{t}=-\frac{1}{2}\log^{2}(z),
\end{align*}
which follows from \cite[Equation (2.13)]{deninger}\footnote{Deninger requires it with $\alpha>0, \beta>0$, however, it is easily seen to hold for Re$(\a)>0$ and Re$(\b)>0$ as well.}
\begin{align*}
	\int_{0}^{\infty}(e^{-\beta t}-e^{-\alpha t})(\gamma+\log(t))\frac{dt}{t}=\frac{1}{2}\left(\log^{2}(\beta)-\log^{2}(\alpha)\right).
\end{align*}
Thus
	\begin{align}\label{befwat}
	\psi_1(z)&={1\over 2} \log ^2(z)-\g(\psi(z)-\log(z))-\int_{0}^{\infty}e^{-zt}\left({1\over t}-\frac{1}{1-e^{-t}} \right){\log(t)}\, dt,
\end{align}
where we employed the well-known result \cite[p.~903, Formula \textbf{8.361.8}]{gr} that for Re$(z)>0$,
\begin{equation*}
\psi(z)=\log(z)+\int_{0}^{\infty}e^{-zt}\left(\frac{1}{t}-\frac{1}{1-e^{-t}}\right)\, dt.
\end{equation*}
We now find the asymptotic expansion of the integral on the right-hand side of \eqref{befwat}, that is, of
\begin{align*}
	I:=\int_{0}^{\infty}e^{-zt}f(t)\, dt,
\end{align*}
where
\begin{equation*}
f(t):=\left({1\over t}-\frac{1}{1-e^{-t}} \right) \log(t),
\end{equation*}
by applying Theorem \ref{Watsongen}. To that end, observe that for $|t|<2\pi$,
 \begin{align*}
	f(t)& = \log(t)\left({1\over t}-\frac{1}{1-e^{-t}} \right) \nonumber \\
	&=-{\log(t)\over t} \left(\frac{te^t}{e^{t}-1}-1 \right) \nonumber \\
	&=-{\log(t)\over t} \left(\sum_{n=0}^{\infty }{B_n(1)t^n\over n!}-1 \right) \nonumber \\
	&={\log(t)} \sum_{n=0}^{\infty }{(-1)^nB_{n+1}t^n\over (n+1)!},
\end{align*}
where $B_n(x)$ are Bernoulli polynomials. 
Thus, with $P_n(x)=x$, $\l_n=n+1$, $a_n={(-1)^nB_{n+1}\over (n+1)!}$, $n\geq 0$, all of the hypotheses of Theorem \ref{Watsongen} are satisfied, and hence
\begin{align*}
	I&\sim\sum_{n=0}^{\infty}\frac{(-1)^nB_{n+1}}{(n+1)z^{n+1}} \left(\psi(n+1)-\log(z)\right) \nonumber\\
	&= \sum_{n=1}^{\infty}{(-1)^{n-1}B_{n}\over nz^{n}} \left(-\g+\sum_{k=1}^{n-1}{1\over k}-\log(z) \right) ,
\end{align*}
where we used the elementary fact $\psi(n)=-\g+\sum_{k=1}^{n-1}\frac{1}{k}$. 
 Inserting this asymptotic expansion of $I$ in \eqref{befwat} along with that of $\psi(z)$, namely, for $|\arg z|<\pi$,
\begin{align}\label{asypsi}
\psi(z) 
&=\log(z)-\frac{1}{2z}-\sum_{n=1}^{\infty}\frac{B_{2n}}{2nz^{2n}},
\end{align}
as $z\to \infty$, we arrive at 
\begin{align*}
	\psi_1(z)&\sim{1\over 2} \log ^2(z)-\g\left(-\frac{1}{2z}-\sum_{n=1}^{\infty}\frac{B_{2n}}{2nz^{2n}}\right)+\sum_{n=1}^{\infty}{(-1)^{n}B_{n}\over nz^{n}} \left(-\g+\sum_{j=1}^{n-1}{1\over j}-\log(z) \right) \nonumber\\
	&={1\over 2} \log ^2(z)+\sum_{n=1}^{\infty}{(-1)^{n}B_{n}\over nz^{n}} \left(\sum_{j=1}^{n-1}{1\over j}-\log(z) \right)\nonumber\\
	&={1\over 2}\log^2(z)-{1 \over 2z}\log z+\sum_{n=1}^{\infty}{B_{2n}\over 2nz^{2n}} \left(\sum_{j=1}^{2n-2}{1\over j}+{1\over 2n-1}-\log(z) \right)
	\end{align*}
	using the well-known facts $B_1=-1/2$ and $B_{2n-1}=0, n>1$. This proves \eqref{asymptotic-psi1_eqn} for $|\arg(z)|<\pi/2$. 
	
	To extend it to $|\arg(z)|<\pi$, we use the analogue of Theorem \ref{watsextended} containing a logarithmic factor in the integrand of the concerned integral, which practically changes none of the hypotheses in the statement of Theorem \ref{watsextended} and its proof\footnote{Temme \cite[p.~15]{temme2015} refers to Olver \cite[p.~114]{olver} for a proof.} since $\log(t)=O\left(t^{\epsilon}\right)$ as $t\to\infty$ for any $\epsilon>0$. We apply it with $\alpha=-\pi/2$ and $\beta=\pi/2$. It shows that the expansion in \eqref{asymptotic-psi1_eqn} holds for $|\arg(z)|<\pi$.
\end{proof}

Our next result is a new analogue of Kloosterman's result for $\psi(x)$ \cite[p.~24-25]{titch}.
\begin{theorem}\label{kloosterman-type}
Let $|\arg(z)|<\pi$. Let $\psi_1$ be defined in \eqref{psi1z}. For $0<c=\textup{Re}(s)<1$,
\begin{align}\label{imt}
	\psi_1(z+1)-\frac{1}{2}\log^{2}(z)=\frac{1}{2\pi i}\int_{(c)}\frac{\pi\zeta(1-s)}{\sin(\pi s)}\left(\gamma-\log(z)+\psi(s)\right)z^{-s}\, ds.
	\end{align}
\end{theorem}
\begin{proof}
We first prove the result for $z>0$ and later extend it to $|\arg(z)|<\pi$ by analytic continuation. 
	Using \eqref{psi1lang}, we have
	\begin{align}\label{aftpar0}
	\psi_1(z)-\frac{1}{2}\log^{2}(z)
	&=\int_{0}^{\infty}e^{-zt}\left(\gamma+\log(t)\right)\, dt+\int_{0}^{\infty}\left(\frac{1}{e^{t}-1}-\frac{1}{t}\right)e^{-zt}\left(\gamma+\log(t)\right)\, dt\nonumber\\
	&=-\frac{\log(z)}{z}+\int_{0}^{\infty}\left(\frac{1}{e^{t}-1}-\frac{1}{t}\right)e^{-zt}\left(\gamma+\log(t)\right)\, dt,
	\end{align}
	where in the last step we employed \cite[p.~573, formula \textbf{4.352.1}]{gr}
	\begin{equation}\label{par1}
	\int_{0}^{\infty}t^{s-1}e^{-zt}\left(\gamma+\log(t)\right)\, dt=\frac{\Gamma(s)}{z^{s}}\left(\gamma-\log(z)+\psi(s)\right)\hspace{10mm}(\textup{Re}(s)>0)
\end{equation}	
with $s=1$ and the fact that $\psi(1)=-\gamma$.
	 We now evaluate the integral on the right-hand side of \eqref{aftpar0} by means of Parseval's formula for Mellin transforms \cite[p.~83, Equation (3.1.13)]{kp}. 
For $0<\textup{Re}(s)<1$, we have \cite[p.~23, Equation (2.7.1)]{titch}
\begin{equation}\label{gammazeta}
\int_{0}^{\infty}t^{s-1}\left(\frac{1}{e^{t}-1}-\frac{1}{t}\right)\, dt=\Gamma(s)\zeta(s).
\end{equation} 
Therefore, along with \eqref{par1} and the equation given above, for $0<c=\textup{Re}(s)<1$, Parseval's formula \eqref{parseval-1} implies
\begin{align}\label{aftpar}
\int_{0}^{\infty}\left(\frac{1}{e^{t}-1}-\frac{1}{t}\right)e^{-zt}\left(\gamma+\log(t)\right)\, dt
&=\frac{1}{2\pi i}\int_{(c)}\Gamma(s)\zeta(s)\frac{\Gamma(1-s)}{z^{1-s}}\left(\gamma-\log(z)+\psi(1-s)\right)\, ds\nonumber\\
&=\frac{1}{2\pi iz}\int_{(c)}\frac{\pi\zeta(s)}{\sin(\pi s)}\left(\gamma-\log(z)+\psi(1-s)\right)z^{s}\, ds\nonumber\\
&=\frac{1}{2\pi i}\int_{(c')}\frac{\pi\zeta(1-s)}{\sin(\pi s)}\left(\gamma-\log(z)+\psi(s)\right)z^{-s}\, ds,
\end{align}
where $0<c'<1$. In the second step, we used the reflection formula $\G(s)\G(1-s)=\pi/\sin(\pi s),s\notin\mathbb{Z}$, 
and in the last step, we replaced $s$ by $1-s$. 

Now \eqref{imt} follows by substituting \eqref{aftpar} in \eqref{aftpar0} and using the fact \cite[Equation (8.3)]{dilcher}
\begin{equation}\label{functionalpsi1}
\psi_1(z)=\psi_1(z+1)-\frac{\log(z)}{z}.
\end{equation}
This proves \eqref{imt} for $z>0$. The result is easily seen to be true for any complex $z$ such that $|\arg(z)|<\pi$ by analytic continuation with the help of \eqref{psi1z}, elementary bounds on the Riemann zeta function, Stirling's formula \eqref{strivert} and the corresponding estimate for $\psi(s)$. 
\end{proof}

In the next theorem, we give a closed-form evaluation of an infinite series of integrals in terms of the digamma function and $\psi_1(z)$. This theorem is an analogue of \eqref{dgkmresult} and will play a fundamental role in the proof of Theorem \ref{loglamb}. 
\begin{theorem}\label{analoguedgkm}
For $\textup{Re}(w)>0$, we have
\begin{align}\label{analoguedgkm_eqn}
4\sum_{m=1}^{\infty}\int_{0}^{\infty}\frac{u\cos(u)\log(u/w)}{u^2+(2\pi mw)^2}\, du
&=\psi_1(iw)-\frac{1}{2}\log^{2}(iw)+\psi_1(-iw)-\frac{1}{2}\log^{2}(-iw)+\frac{\pi}{2w}\nonumber\\
&\quad+\gamma\left(\psi(iw)+\psi(-iw)-2\log(w)\right).
\end{align}
\end{theorem} 
\begin{proof}
Using Theorem \ref{kloosterman-type}, once with $x=iw$, and again with $x=-iw$ and adding the respective sides, for $0<c=\textup{Re}(s)<1$, we obtain
\begin{align}\label{2tkloosterman}
&\psi_1(iw+1)-\frac{1}{2}\log^{2}(iw)+\psi_1(-iw+1)-\frac{1}{2}\log^{2}(-iw)\nonumber\\
&=\frac{1}{2\pi i}\int\limits_{(c)}\frac{\pi\zeta(1-s)}{\sin(\pi s)}\left(\gamma-\log(iw)+\psi(s)\right)(iw)^{-s}ds+\frac{1}{2\pi i}\int\limits_{(c)}\frac{\pi\zeta(1-s)}{\sin(\pi s)}\left(\gamma-\log(-iw)+\psi(s)\right)(-iw)^{-s}ds\nonumber\\
&=\frac{1}{2\pi i}\int\limits_{(c)}\frac{2\pi\zeta(1-s)}{\sin(\pi s)}\left(\gamma+\psi(s)\right)\cos\left(\frac{\pi s}{2}\right)w^{-s}\, ds-I_2,
\end{align}
where
\begin{align}
I_2:=\frac{1}{2\pi i}\int_{(c)}\frac{\pi\zeta(1-s)}{\sin(\pi s)}\left(e^{-\frac{i\pi s}{2}}\log(iw)+e^{\frac{i\pi s}{2}}\log(-iw)\right)w^{-s}\, ds.\label{i2}
\end{align}
Next, using the series expansions of the exponential functions  $e^{\frac{i\pi s}{2}}$ and $e^{-\frac{i\pi s}{2}}$ and splitting them according to $n$ even and $n$ odd, it is easily seen that
\begin{equation}\label{elem_idty}
e^{-\frac{i\pi s}{2}}\log(iw)+e^{\frac{i\pi s}{2}}\log(-iw)=2\log(w)\cos\left(\frac{\pi s}{2}\right)+\pi\sin\left(\frac{\pi s}{2}\right).
\end{equation}
Hence from \eqref{2tkloosterman}, \eqref{i2} and \eqref{elem_idty}, we see that
\begin{align}\label{j1-j2}
\psi_1(iw+1)-\frac{1}{2}\log^{2}(iw)+\psi_1(-iw+1)-\frac{1}{2}\log^{2}(-iw)=:J_1-J_2,
\end{align}
where
\begin{align*}
J_1&:=\frac{1}{2\pi i}\int_{(c)}\frac{\pi\zeta(1-s)}{\sin\left(\frac{\pi s}{2}\right)}\left(\gamma+\psi(s)-\log(w)\right)w^{-s}\, ds,\\
J_2&:=\frac{1}{2\pi i}\int_{(c)}\frac{\pi^2\zeta(1-s)}{2\cos\left(\frac{\pi s}{2}\right)}w^{-s}\, ds.
\end{align*}
Using the functional equation of $\zeta(s)$ \cite[p.~13, Equation (2.1.1)]{titch}
\begin{align}\label{zetaasym}
	\zeta(s) = 2^s \pi^{s-1} \Gamma(1-s) \zeta(1-s) \sin\left(  \frac{1}{2}\pi s \right),
\end{align}
with $s$ replaced by $1-s$, we have 
\begin{align}\label{j2eval}
J_2=\frac{\pi^2}{2\pi i}\int_{(c)}\Gamma(s)\zeta(s)(2\pi w)^{-s}\, ds
=\pi^{2}\left(\frac{1}{e^{2\pi w}-1}-\frac{1}{2\pi w}\right),
\end{align}
where the last step follows from \eqref{gammazeta}. Again using \eqref{zetaasym} with $s$ replaced by $1-s$, we observe that
\begin{align*}
J_1=\frac{1}{2\pi i}\int_{(c)}2\pi\Gamma(s)\zeta(s)\cot\left(\frac{\pi s}{2}\right)\left(\gamma+\psi(s)-\log(w)\right)(2\pi w)^{-s}\, ds.
\end{align*}
Now it is important to observe that shifting the line of integration from Re$(s)=c$, $0<c<1$ to Re$(s)=d,\ 1<d<2$ does not introduce any pole of the integrand. So consider the rectangular contour $[c-iT, d-iT], [d-iT, d+iT], [d+iT, c+iT]$ and $[c+iT, c-iT]$.  By Cauchy's residue theorem and the fact that the integral along the horizontal segments of the contour tend to zero as $T\to\infty$, as can be seen from \eqref{strivert}, elementary bounds of the zeta function and the fact \cite[Equation (15)]{gonek} $\psi(s)=\log(s)+O(1/|s|)$, it is seen that
\begin{align}\label{before_splitting}
J_1&=\frac{1}{2\pi i}\int_{(d)}2\pi\Gamma(s)\zeta(s)\cot\left(\frac{\pi s}{2}\right)\left(\gamma+\psi(s)-\log(w)\right)(2\pi w)^{-s}\, ds\nonumber\\
&=\sum_{m=1}^{\infty}\frac{1}{2\pi i}\int_{(d)}2\pi\Gamma(s)\cot\left(\frac{\pi s}{2}\right)\left(\gamma+\psi(s)-\log(w)\right)(2\pi mw)^{-s}\, ds\nonumber\\
&=2\pi\sum_{m=1}^{\infty}\left\{\frac{\left(\gamma-\log(w)\right)}{2\pi i}\int_{(d)}\Gamma(s)\cot\left(\frac{\pi s}{2}\right)(2\pi mw)^{-s}\, ds+\frac{1}{2\pi i}\int_{(d)}\Gamma(s)\psi(s)\cot\left(\frac{\pi s}{2}\right)(2\pi mw)^{-s}\, ds\right\},
\end{align}
where in the last step, we used the series representation for $\zeta(s)$ and then interchanged the order of summmation and integration which is valid because of absolute and uniform convergence. We now find convenient representations for the two line integrals.

From \cite[Lemma 4.1]{DGKM}, for $0<\textup{Re}(s)=(c_1)<2$ and Re$(z)>0$,
\begin{align*}
\frac{1}{2\pi i}\int_{(c_1)}\Gamma(s)\cot\left(\frac{\pi s}{2}\right)z^{-s}\, ds=\frac{2}{\pi}\int_{0}^{\infty}\frac{u\cos(u)}{u^2+z^2}\, du.
\end{align*}
Hence invoking the above result with $z=2\pi mw$, we find that
\begin{align}\label{while_splitting}
\sum_{m=1}^{\infty}\frac{1}{2\pi i}\int_{(d)}\Gamma(s)\cot\left(\frac{\pi s}{2}\right)(2\pi mw)^{-s}\, ds&=\frac{2}{\pi}\sum_{m=1}^{\infty}\int_{0}^{\infty}\frac{u\cos(u)}{u^2+(2\pi mw)^2}\, du\nonumber\\
&=\frac{1}{\pi}\left(\log(w)-\frac{1}{2}\left(\psi(iw)+\psi(-iw)\right)\right),
\end{align}
where in the last step, we employed \eqref{dgkmresult}.

Here, it is important to note that the series representation of $J_1$ in \eqref{before_splitting} and the convergence of the series in \eqref{while_splitting} together imply that the series
\begin{align*}
\sum_{m=1}^{\infty}\frac{1}{2\pi i}\int_{(d)}\Gamma(s)\psi(s)\cot\left(\frac{\pi s}{2}\right)(2\pi mw)^{-s}\, ds
\end{align*}
converges too. 

Using Parseval's formula, we now suitably transform the line integral in the above equation into an integral of a real variable. From \cite[p.~536, Formula \textbf{2.6.32.2}]{prudnikov1}, for $0<c=\textup{Re}(s)<1$,
\begin{equation*}
\int_{0}^{\infty}u^{s-1}\cos(u)\log(u)\, du=\Gamma(s)\left(\psi(s)\cos\left(\frac{\pi s}{2}\right)-\frac{\pi}{2}\sin\left(\frac{\pi s}{2}\right)\right),
\end{equation*}
whereas for $-1<\textup{Re}(s)<1$ and Re$(z)>0$, we have
\begin{align*}
\int_{0}^{\infty}u^{s-1}\frac{u}{u^2+z^2}\, du=\frac{\pi}{2}z^{s-1}\sec\left(\frac{\pi s}{2}\right). 
\end{align*}
Hence from \eqref{parseval-1}, for $0<\textup{Re}(s)=c<1$, we have
\begin{equation*}
\frac{1}{2\pi i}\int_{(c)}\Gamma(s)\psi(s)\cot\left(\frac{\pi s}{2}\right)(2\pi mw)^{-s}\, ds=\frac{2}{\pi}\int_{0}^{\infty}\frac{u\cos(u)\log(u)}{u^2+(2\pi mw)^2}\, du+\frac{\pi}{2}e^{-2\pi mw}.
\end{equation*}
Now again, shifting the line of integration from Re$(s)=c$ to Re$(s)=d, 1<d<2$, noting that there is no pole of the integrand, applying Cauchy's residue theorem and making use of the fact that the integrals along the horizontal segments tend to zero as the height of the contour tends to $\infty$, we see that 
 \begin{align*}
\frac{1}{2\pi i}\int_{(c)}\Gamma(s)\psi(s)\cot\left(\frac{\pi s}{2}\right)(2\pi mw)^{-s}\, ds=
\frac{1}{2\pi i}\int_{(d)}\Gamma(s)\psi(s)\cot\left(\frac{\pi s}{2}\right)(2\pi mw)^{-s}\, ds.
\end{align*}
Hence
\begin{align}\label{after_splitting}
\sum_{m=1}^{\infty}\frac{1}{2\pi i}\int_{(d)}\Gamma(s)\psi(s)\cot\left(\frac{\pi s}{2}\right)(2\pi mw)^{-s}\, ds=\sum_{m=1}^{\infty}\left\{\frac{2}{\pi}\int_{0}^{\infty}\frac{u\cos(u)\log(u)}{u^2+(2\pi mw)^2}\, du+\frac{\pi}{2}e^{-2\pi mw}\right\}.
\end{align}
Therefore, substituting \eqref{after_splitting} and \eqref{while_splitting} in \eqref{before_splitting}, we get
\begin{align}\label{j1eval}
J_1&=2\pi\Bigg[\frac{\left(\gamma-\log(w)\right)}{2\pi}\left(2\log(w)-\left(\psi(iw)+\psi(-iw)\right)\right)+\sum_{m=1}^{\infty}\Bigg\{\frac{2}{\pi}\int_{0}^{\infty}\frac{u\cos(u)\log(u)}{u^2+(2\pi mw)^2}\, du+\frac{\pi}{2}e^{-2\pi mw}\Bigg\}\Bigg]\nonumber\\
&=2(\gamma-\log(w))\left(\log(w)-\frac{1}{2}\left(\psi(iw)+\psi(-iw)\right)\right)+4\sum_{m=1}^{\infty}\int_{0}^{\infty}\frac{u\cos(u)\log(u)}{u^2+(2\pi mw)^2}\, du+\frac{\pi^2}{e^{2\pi w}-1}.
\end{align}
Now from \eqref{j1-j2}, \eqref{j2eval} and \eqref{j1eval}, we have
\begin{align*}
&\psi_1(iw+1)-\frac{1}{2}\log^{2}(iw)+\psi_1(-iw+1)-\frac{1}{2}\log^{2}(-iw)\nonumber\\
&=2(\gamma-\log(w))\left(\log(w)-\frac{1}{2}\left(\psi(iw)+\psi(-iw)\right)\right)+4\sum_{m=1}^{\infty}\int_{0}^{\infty}\frac{u\cos(u)\log(u)}{u^2+(2\pi mw)^2}\, du+\frac{\pi^2}{e^{2\pi w}-1}\nonumber\\
&\quad-\pi^{2}\left(\frac{1}{e^{2\pi w}-1}-\frac{1}{2\pi w}\right).
\end{align*}
Now using \eqref{functionalpsi1} twice and using the elementary fact $\log(iw)-\log(-iw)=\pi i$ for Re$(w)>0$, we are led to
\begin{align}\label{f1}
4\sum_{m=1}^{\infty}\int_{0}^{\infty}\frac{u\cos(u)\log(u)}{u^2+(2\pi mw)^2}\, du
&=\psi_1(iw)-\frac{1}{2}\log^{2}(iw)+\psi_1(-iw)-\frac{1}{2}\log^{2}(-iw)+\frac{\pi}{2w}\nonumber\\
&\quad+\left(\gamma-\log(w)\right)\left(\psi(iw)+\psi(-iw)-2\log(w)\right).
\end{align}
Lastly, employing \eqref{dgkmresult} again with $w$ replaced by $2\pi w$, we see that
\begin{align}\label{f2}
-4\log(w)\sum_{m=1}^{\infty}\int_{0}^{\infty}\frac{u\cos(u)}{u^2+(2\pi mw)^2}\, du=\log(w)\left(\psi(iw)+\psi(-iw)-2\log(w)\right).
\end{align}
Finally, adding the respective sides of \eqref{f1} and \eqref{f2}, we are led to \eqref{analoguedgkm_eqn}. This completes the proof.
\end{proof}
\section{A new representation for $\left.\frac{\partial^2}{\partial b^2}E_{2, b}(z)\right|_{b=1}$}\label{mittag}

The two-variable Mittag-Leffler function $E_{\alpha, \beta}(z)$, introduced by Wiman \cite{wiman}, is defined by
\begin{align}\label{2varmldef}
E_{\alpha, \beta}(z):=\sum_{k=0}^{\infty}\frac{z^k}{\G(\alpha k+\beta)}\hspace{10mm}(\textup{Re}(\alpha)>0, \textup{Re}(\beta)>0).
\end{align}
There is an extensive literature on these Mittag-Leffler functions, see, for example, \cite{gkmr}, and the references therein. Yet, closed-form expressions exist only for the first derivatives of $E_{\alpha, \beta}(z)$ with respect to the parameters $\alpha$ and $\beta$. The reader is referred to a recent paper of Apelblat \cite{apelblat} for a collection of such evaluations all of which involve the $\textup{Shi}(z)$ and $\textup{Chi}(z)$ functions defined in \eqref{shichi}.

From \eqref{2varmldef}, it is easy to see that \cite[Equation (98)]{apelblat}
\begin{align}\label{relation}
\frac{\partial^2}{\partial\beta^2}E_{\alpha, \beta}(t)=\sum_{k=0}^{\infty}\frac{\psi^{2}(\alpha k+\beta)-\psi'(\alpha k+\beta)}{\G(\alpha k+\beta)}t^k.
\end{align}
However, there are no closed-form evaluations known for the above series.
 In what follows, we establish a new result which transforms $\left.\frac{\partial^2}{\partial\beta^2}E_{2, \beta}(t)\right|_{\beta=1}$ into a suitable series which is absolutely essential in proving Theorem \ref{loglamb}. Before we embark upon the proof though, we need a few lemmas concerning the \emph{hyperbolic sine and cosine integrals} $\textup{Shi}(z)$ and $\textup{Chi}(z)$ defined by \cite[p.~150, Equation (6.2.15), (6.2.16)]{NIST}
\begin{align}\label{shichi}
\mathrm{Shi}(z):=\int_0^z\frac{\sinh(t)}{t}\ dt,\hspace{3mm}
\mathrm{Chi}(z):=\gamma+\log(z)+\int_0^z\frac{\cosh(t)-1}{t}\ dt.
\end{align} The first lemma involving these functions was established in \cite[Lemma 9.1]{dkk}. Here, and throughout the sequel, we use the following notation for brevity.
\begin{align*}
			\left(\sinh\operatorname{Shi}-\cosh\operatorname{Chi}\right)(w):=\sinh(w) \operatorname{Shi} (w)- \cosh(w) \operatorname{Chi} (w).    
		\end{align*}
\begin{lemma}\label{minusonelemma}
Let $\mathrm{Re}(w)>0$. 
Then
\begin{equation*}
\int_{0}^{\infty}\frac{t\cos t\, dt}{t^2+w^2}=\left(\sinh\operatorname{Shi}-\cosh\operatorname{Chi}\right)(w).
\end{equation*}
\end{lemma}
We require another result from \cite[Lemma 3.2]{DGKM}.
\begin{lemma}\label{sumpsi}
For \textup{Re}$(w)>0$,
\begin{align*}
\sum_{k=0}^\infty\frac{\psi(2k+1)}{\Gamma\left(2k+1\right)}w^{2k}&=\left(\sinh\operatorname{Shi}-\cosh\operatorname{Chi}\right)(w)+\log(w)\cosh(w).
\end{align*}
\end{lemma}
We are now ready to prove our result on evaluating the second derivative of the Mittag-Leffler function $E_{2,b}(z)$ with respect to the parameter $b$ at $b=1$.
\begin{theorem}\label{dzhrbashyan}
Let $b>0$ and $\textup{Re}(w)>0$. Let the two-variable Mittag-Leffler function $E_{2,b}(z)$ be defined in \eqref{2varmldef}. Then
\begin{align}\label{dzhrbashyan_eqn}
\left.\frac{\partial^2}{\partial b^2}E_{2, b}(w^2)\right|_{b=1}=\log^{2}(w)\cosh(w)+2\int_{0}^{\infty}\frac{u\cos(u)\log(u)\, du}{u^2+w^2}.
\end{align}
\end{theorem}
\begin{proof}
To prove this result, we first  apply a result of Dzhrbashyan \cite[p.~130, Equation (2.12)]{dzhrbashyan} that will be used in the proof. It states that for $0\leq\arg(z)<\pi$ or $-\pi<\arg(z)\leq0$, one has\footnote{It is to be noted that Dzhrbashyan uses a different notation for the two-variable Mittag-Leffler function in his book, namely, $E_{\rho}(z;b)=\sum_{k=0}^{\infty}z^k/\G(b+k\rho^{-1})$. See \cite[p.~117]{dzhrbashyan}. We have used the contemporary notation here.}
\begin{align}\label{eqn_Dzhrbashyan}
E_{2,b}(z)=\frac{1}{2}z^{\frac{1}{2}(1-b)}\left\{e^{\sqrt{z}}+e^{\mp i\pi(1-b)-\sqrt{z}}\right\}+\frac{\sin(\pi b)}{2\pi}\int_{0}^{\infty}\frac{e^{\pm i\left(\sqrt{t}-\frac{\pi}{2}(1-b)\right)}}{t+z}t^{\frac{1}{2}(1-b)}\, dt,
\end{align}
where the upper or lower signs are taken, respectively, for $0\leq\arg(z)<\pi$ or $-\pi<\arg(z)\leq0$. 

We first assume $0\leq\arg(z)<\pi$. Differentiating \eqref{eqn_Dzhrbashyan} twice with respect to $b$ and simplifying leads to
\begin{align*}
\frac{\partial^2}{\partial b^2}E_{2, b}(z)&=\frac{1}{2}\Bigg[-\frac{1}{2}\log(z)\left\{-\frac{1}{2}z^{\frac{1}{2}(1-b)}\log(z)\left(e^{\sqrt{z}}+e^{-i\pi(1-b)-\sqrt{z}}\right)+i\pi z^{\frac{1}{2}(1-b)}e^{i\pi(b-1)-\sqrt{z}}\right\}\nonumber\\
&\quad+i\pi\sqrt{z}e^{-\sqrt{z}}z^{-\frac{b}{2}}e^{i\pi(b-1)}\left(i\pi-\frac{1}{2}\log(z)\right)\Bigg]-\frac{\pi\sin(\pi b)}{2}\int_{0}^{\infty}\frac{e^{i\left(\sqrt{t}-\frac{\pi}{2}(1-b)\right)}}{t+z}t^{\frac{1}{2}(1-b)}\, dt\nonumber\\
&\quad+\frac{1}{2}\cos(\pi b)\int_{0}^{\infty}\frac{e^{i\left(\sqrt{t}-\frac{\pi}{2}(1-b)\right)}}{t+z}t^{\frac{1}{2}(1-b)}(i\pi-\log(t))\, dt\nonumber\\
&\quad+\frac{1}{8\pi}\sin(\pi b)\int_{0}^{\infty}\frac{e^{i\left(\sqrt{t}-\frac{\pi}{2}(1-b)\right)}}{t+z}t^{\frac{1}{2}(1-b)}(i\pi-\log(t))^{2}\, dt,
\end{align*}
so that
\begin{align}\label{extm0}
\left.\frac{\partial^2}{\partial b^2}E_{2, b}(z)\right|_{b=1}=\frac{1}{4}\log^{2}(z)\cosh(\sqrt{z})-\frac{i\pi}{2}e^{-\sqrt{z}}\log(z)-\frac{\pi^2}{2}e^{-\sqrt{z}}-\frac{1}{2}\int_{0}^{\infty}\frac{e^{i\sqrt{t}}(i\pi-\log(t))}{t+z}\, dt.
\end{align}
Now employing the change of variable $t=u^2$, we have
\begin{align}\label{aftchange}
\int_{0}^{\infty}\frac{e^{i\sqrt{t}}(i\pi-\log(t))}{t+z}\, dt&=2\int_{0}^{\infty}\frac{ue^{iu}(i\pi-2\log(u))}{u^2+z}\, du\nonumber\\
&=2\pi i\int_{0}^{\infty}\frac{u\cos(u)}{u^2+z}\, du-2\pi\int_{0}^{\infty}\frac{u\sin(u)}{u^2+z}\, du-4\int_{0}^{\infty}\frac{ue^{iu}\log(u)}{u^2+z}\, du.
\end{align}
Next, from Lemma \ref{minusonelemma},
\begin{align}\label{minus}
\int_{0}^{\infty}\frac{u\cos(u)}{u^2+z}\, du=\left(\sinh\operatorname{Shi}-\cosh\operatorname{Chi}\right)(\sqrt{z}),
\end{align}
whereas from \cite[p.~65, Formula (15)]{batemanmp},
\begin{equation}\label{plus}
\int_{0}^{\infty}\frac{u\sin(u)}{u^2+z}\, du=\frac{\pi}{2}e^{-\sqrt{z}}.
\end{equation}
Moreover, 
\begin{align}\label{intm}
\int_{0}^{\infty}\frac{ue^{iu}\log(u)}{u^2+z}\, du&=\int_{0}^{\infty}\frac{u\cos(u)\log(u)}{u^2+z}\, du+i\int_{0}^{\infty}\frac{u\sin(u)\log(u)}{u^2+z}\, du,
\end{align}
Now from \cite[p.~537, Formula \textbf{2.6.32.8}]{prudnikov1}\footnote{The formula given in the book has a typo. We have corrected it here.},
\begin{align}\label{usinlog}
\int_{0}^{\infty}\frac{u\sin(u)\log(u)}{u^2+z}\, du=\frac{\pi}{2}\left\{e^{-\sqrt{z}}\log(\sqrt{z})-\frac{1}{2}\left(e^{\sqrt{z}}\textup{Ei}(-\sqrt{z})+e^{-\sqrt{z}}\textup{Ei}(\sqrt{z})\right)\right\},
\end{align}
where $\textup{Ei}(x)$ is the exponential integral given in  \cite[p.~788]{Prudnikov} for $x>0$ by $\textup{Ei}(x):=\int_{-\infty}^{x}e^{t}/t\, dt,$ or, as is seen from \cite[p.~1]{je}, by $\textup{Ei}(-x):=-\int_{x}^{\infty}e^{-t}/t\, dt$. However, from \cite[p.~395, Formula \textbf{2.5.9.12}]{prudnikov1},
\begin{equation}\label{one}
\int_{0}^{\infty}\frac{t\cos t\, dt}{t^2+z}=-\frac{1}{2}\left(e^{\sqrt{z}}\textup{Ei}(-\sqrt{z})+e^{-\sqrt{z}}\textup{Ei}(\sqrt{z})\right).
\end{equation}
Hence along with Lemma \ref{minusonelemma}, \eqref{usinlog} and \eqref{one} imply
\begin{align}\label{usinlog1}
\int_{0}^{\infty}\frac{u\sin(u)\log(u)}{u^2+z}\, du=\frac{\pi}{2}\left\{e^{-\sqrt{z}}\log(\sqrt{z})+\left(\sinh\operatorname{Shi}-\cosh\operatorname{Chi}\right)(\sqrt{z})\right\}.
\end{align}
Therefore, from \eqref{aftchange}, \eqref{minus}, \eqref{plus}, \eqref{intm} and \eqref{usinlog1}, we arrive at
\begin{align}\label{extm}
\int_{0}^{\infty}\frac{e^{i\sqrt{t}}(i\pi-\log(t))}{t+z}\, dt&=-\pi^2e^{-\sqrt{z}}-2\pi i e^{-\sqrt{z}}\log(\sqrt{z})-4\int_{0}^{\infty}\frac{u\cos(u)\log(u)}{u^2+z}\, du.
\end{align}  
Finally, substituting \eqref{extm} in \eqref{extm0}, we arrive at 
\begin{align*}
\left.\frac{\partial^2}{\partial b^2}E_{2, b}(z)\right|_{b=1}=\log^2(\sqrt{z})\cosh(\sqrt{z})+2\int_{0}^{\infty}\frac{u\cos(u)\log(u)}{u^2+z}\, du.
\end{align*}
Now let $z=w^2$ to arrive at \eqref{dzhrbashyan_eqn} for $0\leq\arg(w)<\frac{\pi}{2}$. Simiarly starting with $-\pi<\arg(z)\leq0$ and proceeding exactly as above and putting $z=w^2$ at the end, we again obtain \eqref{dzhrbashyan_eqn} for $-\frac{\pi}{2}<\arg(w)\leq0$. Hence we see that \eqref{dzhrbashyan_eqn} actually holds for Re$(w)>0$.
\end{proof}

\section{Proof of the main results}
This section is devoted to proving the transformation for the Lambert series for logarithm given in Theorem \ref{loglamb}, the asymptotic expansion of this series in Theorem \ref{logy0}, and a mean value theorem in Theorem \ref{moments}. 

\subsection{Proof of Theorem \ref{loglamb}}
The idea is to differentiate both sides of \eqref{maineqn} with respect to $a$ and then let $a\to0$. Define
\begin{align}
F_1(a,y)&:=\frac{d}{da}\sum_{n=1}^{\infty}\sigma_a(n)e^{-ny},\nonumber\\
F_2(a, y)&:=\frac{d}{da}\left\{\frac{1}{2}\left(\left(\frac{2\pi}{y}\right)^{1+a}\mathrm{cosec}\left(\frac{\pi a}{2}\right)+1\right)\zeta(-a)-\frac{1}{y}\zeta(1-a)\right\},\nonumber\\
F_3(a, y)&:=\frac{d}{da}\left\{\frac{2\pi}{y\sin\left(\frac{\pi a}{2}\right)}\sum_{n=1}^\infty \sigma_{a}(n)\Bigg(\tfrac{(2\pi n)^{-a}}{\Gamma(1-a)} {}_1F_2\left(1;\tfrac{1-a}{2},1-\tfrac{a}{2};\tfrac{4\pi^4n^2}{y^2} \right) -\left(\tfrac{2\pi}{y}\right)^{a}\hspace{-5pt}\cosh\left(\tfrac{4\pi^2n}{y}\right)\Bigg)\right\}\label{f3ay},
\end{align}
and let
\begin{align}\label{g1g2g3}
G_1(y)=\lim_{a\to0}F_1(a, y),\hspace{5mm}G_2(y)=\lim_{a\to0}F_2(a, y),\hspace{5mm}G_3(y)=\lim_{a\to0}F_3(a, y).
\end{align}
Clearly, from \eqref{maineqn},
\begin{align}\label{diffid}
G_1(y)+G_2(y)=G_3(y).
\end{align}
Using \eqref{equivalentlambert}, it is readily seen that
\begin{equation}\label{g1eval}
G_1(y)=\sum_{n=1}^{\infty}\frac{\log(n)}{e^{ny}-1}.
\end{equation}
By routine differentiation, 
\begin{align}\label{f2ybeflimit}
F_2(a, y)&=\frac{1}{y}\zeta'(1-a)-\frac{1}{2}\zeta'(-a)-\frac{1}{2}\left(\frac{2\pi}{y}\right)^{1+a}\mathrm{cosec}\left(\frac{\pi a}{2}\right)\zeta'(-a)\nonumber\\
&\quad+\frac{1}{2}\zeta(-a)\left(\frac{2\pi}{y}\right)^{1+a}\mathrm{cosec}\left(\frac{\pi a}{2}\right)\left(\log\left(\frac{2\pi}{y}\right)-\frac{\pi}{2}\cot\left(\frac{\pi a}{2}\right)\right).
\end{align}
Letting $a\to0$ in \eqref{f2ybeflimit} leads to
\begin{align}\label{g2eval}
G_2(y)=\frac{1}{4}\log(2\pi)-\frac{1}{2y}\log^{2}(y)+\frac{\gamma^2}{2y}-\frac{\pi^2}{12y},
\end{align}
which is now proved. Note that
\begin{align}\label{simplim}
G_2(y) &=\lim_{a\to 0} \left(-\frac{1}{2}\zeta'(-a)\right) \nonumber\\
&\quad+ \lim_{a\to 0} \left[\frac{1}{y}\zeta'(1-a)-\frac{\left(2\pi/y\right)^{1+a}}{2\sin\left(\frac{\pi a}{2}\right)}\left\{ \zeta'(-a)-\zeta(-a)\log\left(\frac{2\pi}{y}\right) + \frac{\pi}{2}\zeta(-a)\cot\left(\frac{\pi a}{2}\right)\right\} \right].
\end{align}
We make use of the following well-known Laurent series expansions around $a=0$:
\begin{align}\label{expansions}
\zeta(-a) &= -\frac{1}{2} + \frac{1}{2} \log(2\pi)a + \left[\frac{\gamma^2}{4}-\frac{\pi^2}{48}- \frac{\log^2(2\pi)}{4}+ \frac{\gamma_1}{2} \right]a^2 + O(a^3),\nonumber\\
\zeta'(1-a)&= -\frac{1}{a^2}-\gamma_1-\gamma_2 a+ O(a^2),\nonumber\\
\left(\frac{2\pi}{y}\right)^{1+a}&= \frac{2\pi}{y}+\frac{2\pi}{y}\log\left(\frac{2\pi}{y}\right) a + \frac{\pi}{y} \log^2 \left(\frac{2\pi}{y}\right) a^2 + O(a^3),\nonumber\\
\textup{cosec} \left(\frac{\pi a}{2}\right) &= \frac{2}{\pi a}+ \frac{\pi}{12}a +   O(a^3),\nonumber\\
\cot \left(\frac{\pi a}{2}\right) &= \frac{2}{\pi a}- \frac{\pi}{6}a +   O(a^3).
\end{align}
Hence substituting \eqref{expansions} in \eqref{simplim}, we get
\begin{align*}
G_2(y)&=\tfrac{1}{4}\log(2\pi) + \lim_{a\to 0}\Big[-\tfrac{1}{ya^2}-\tfrac{\gamma_1}{y} + O(a)+ \left\{-\tfrac{2}{ya} -\tfrac{2}{y}\log\left(\tfrac{2\pi}{y}\right)+\left(-\tfrac{\pi^2}{12y}-\tfrac{1}{y}\log^2 \left(\tfrac{2\pi}{y}\right) \right)a + O(a^2) \right\}\nonumber\\
&\qquad \qquad \times \left\{-\tfrac{1}{2a}+ \tfrac{1}{2}\log\left(\tfrac{2\pi}{y}\right) + \left(-\tfrac{\gamma^2}{4}+\tfrac{\pi^2}{16}-\tfrac{\gamma_1}{2}+\tfrac{1}{4}\log^2(2\pi)-\tfrac{1}{2}\log(2\pi)\log\left(\tfrac{2\pi}{y}\right) \right)a + O(a^2) \right\} \Big]\nonumber\\
&= \frac{1}{4}\log(2\pi) + \lim_{a\to 0}\Bigg[\left(-\frac{1}{y}+ \frac{1}{y} \right)\frac{1}{a^2} + \left(-\frac{1}{y}\log\left(\frac{2\pi}{y}\right) + \frac{1}{y}\log\left(\frac{2\pi}{y}\right)\right)\frac{1}{a} \nonumber\\
&\qquad + \left(-\frac{\gamma_1}{y}+ \frac{\gamma^2}{2y}-\frac{\pi^2}{12y}+\frac{\gamma_1}{y}-\frac{1}{2y}\left\{\log^2(2\pi)-2\log(2\pi)\log\left(\frac{2\pi}{y}\right)+ \log^2\left(\frac{2\pi}{y}\right) \right\}\right)+ O(a) \Bigg]\nonumber\\
&= \frac{1}{4}\log(2\pi) + \frac{\gamma^2}{2y} -\frac{\pi^2}{12y} -\frac{\log^2y}{2y}.
\end{align*}
Next, we apply the definitions of $\sigma_a(n)$, ${}_1F_{2}$ and simplify to obtain
\begin{align*}
F_3(a, y)&=\frac{d}{da}\left[\frac{2\pi}{y\sin\left(\frac{\pi a}{2}\right)}\sum_{m, n=1}^\infty\left\{(2\pi n)^{-a}\sum_{k=0}^{\infty}\frac{\left(4\pi^2mn/y\right)^{2k}}{\G(1-a+2k)}-\left(\frac{2\pi m}{y}\right)^{a}\cosh\left(\frac{4\pi^2mn}{y}\right)\right\}\right]\notag\\
&=\frac{2\pi}{y}\sum_{m, n=1}^{\infty}\sum_{k=0}^{\infty}\left(\frac{4\pi^2mn}{y}\right)^{2k}\frac{1}{\sin^{2}\left(\frac{\pi a}{2}\right)}\Bigg[\frac{(2\pi n)^{-a}}{\G(1-a+2k)}\bigg\{\left(\psi(1-a+2k)-\log(2\pi n)\right)\sin\left(\frac{\pi a}{2}\right)\nonumber\\
&\qquad\quad-\frac{\pi}{2}\cos\left(\frac{\pi a}{2}\right)\bigg\}-\frac{1}{\G(2k+1)}\left(\frac{2\pi m}{y}\right)^{a}\left(\log\left(\frac{2\pi m}{y}\right)\sin\left(\frac{\pi a}{2}\right)-\frac{\pi}{2}\cos\left(\frac{\pi a}{2}\right)\right)\Bigg]\nonumber\\
&=\frac{2\pi}{y}\sum_{m, n=1}^{\infty}\sum_{k=0}^{\infty}\left(\frac{4\pi^2mn}{y}\right)^{2k}\frac{A_1(a)-A_2(a)}{\sin^{2}\left(\frac{\pi a}{2}\right)},
\end{align*}
where in the first step we have interchanged the order of summation and differentiation using the fact \cite[p.~35]{dkk} that the series in the definition of  \eqref{f3ay} converges uniformly as long as Re$(a)>-1$, and where
\begin{align*}
\hspace{-3mm}A_1(a)&=A_1(m, n, k; a):=\frac{(2\pi n)^{-a}}{\G(1-a+2k)}\left\{\left(\psi(1-a+2k)-\log(2\pi n)\right)\sin\left(\frac{\pi a}{2}\right)-\frac{\pi}{2}\cos\left(\frac{\pi a}{2}\right)\right\},\\
\hspace{-3mm}A_2(a)&=A_2(m, n, k; a):=\frac{1}{\G(2k+1)}\left(\frac{2\pi m}{y}\right)^{a}\left\{\log\left(\frac{2\pi m}{y}\right)\sin\left(\frac{\pi a}{2}\right)-\frac{\pi}{2}\cos\left(\frac{\pi a}{2}\right)\right\}.
\end{align*}
We now show that
\begin{align}\label{comp_limit}
\lim_{a\to0}\frac{A_1(a)-A_2(a)}{\sin^{2}\left(\frac{\pi a}{2}\right)}
&=\frac{1}{\pi\G(2k+1)}\bigg\{\log\left(\frac{4\pi^2mn}{y}\right)\log\left(\frac{ny}{m}\right)-2\log(2n\pi)\psi(2k+1)+\psi^{2}(2k+1)\nonumber\\
&\qquad\qquad\qquad\quad-\psi'(2k+1)\bigg\}.
\end{align}
Let $L$ denote the limit in the equation given above. The expression inside the limit is of the form $0/0$. By routine calculation, we find that
\begin{align}
A_1'(a)&=\frac{(2\pi n)^{-a}}{\G(1-a+2k)}\left\{-\psi'(1-a+2k)+\frac{\pi^2}{4}+\left(\psi(1-a+2k)-\log(2\pi n)\right)^{2}\right\}\sin\left(\frac{\pi a}{2}\right),\label{a1d1}\\
A_2'(a)&=\frac{\left(2\pi m/y\right)^{a}}{\G(2k+1)}\bigg\{\log\left(\frac{2\pi m}{y}\right)\left(\log\left(\frac{2\pi m}{y}\right)\sin\left(\frac{\pi a}{2}\right)-\frac{\pi}{2}\cos\left(\frac{\pi a}{2}\right)\right)\nonumber\\
&\qquad\qquad\qquad+\frac{\pi}{2}\log\left(\frac{2\pi m}{y}\right)\cos\left(\frac{\pi a}{2}\right)+\frac{\pi^2}{4}\sin\left(\frac{\pi a}{2}\right)\bigg\},\label{a2d1}
\end{align}
where $'$ denotes differentiation with respect to $a$.
Since $A_1'(0)=A_2'(0)=0$, using L'Hopital's rule again, it is seen that
\begin{align*}
L=\frac{A_1''(0)-A_2''(0)}{\pi^2/2},
\end{align*}
Differentiating \eqref{a1d1} and \eqref{a2d1} with respect to $a$, we get
\begin{align*}
A_1''(0)&=\frac{\pi}{2\G(2k+1)}\left\{-\psi'(2k+1)+\frac{\pi^2}{4}+\left(\psi(2k+1)-\log(2\pi n)\right)^{2}\right\},\\
A_2''(0)&=\frac{1}{\G(2k+1)}\left\{\frac{\pi}{2}\log^{2}\left(\frac{2\pi m}{y}\right)+\frac{\pi^3}{8}\right\},
\end{align*}
whence we obtain \eqref{comp_limit} upon simplification. Therefore, from \eqref{g1g2g3} and \eqref{comp_limit}, we deduce that
\begin{align*}
G_3(y)&=\frac{2}{y}\sum_{m, n=1}^{\infty}\sum_{k=0}^{\infty}\frac{\left(4\pi^2mn/y\right)^{2k}}{(2k)!}\bigg\{\log\left(\frac{4\pi^2mn}{y}\right)\log\left(\frac{ny}{m}\right)-2\log(2n\pi)\psi(2k+1)\nonumber\\
&\qquad\qquad\qquad\qquad\qquad\qquad+\psi^{2}(2k+1)-\psi'(2k+1)\bigg\}\nonumber\\
&=\frac{2}{y}\sum_{m, n=1}^{\infty}\Bigg\{\log\left(\frac{4\pi^2mn}{y}\right)\log\left(\frac{ny}{m}\right)\cosh\left(\frac{4\pi^2mn}{y}\right)-2\log(2n\pi)\sum_{k=0}^{\infty}\frac{\psi(2k+1)}{\G(2k+1)}\left(\frac{4\pi^2mn}{y}\right)^{2k}\nonumber\\
&\qquad\qquad\quad+\sum_{k=0}^{\infty}\frac{\psi^{2}(2k+1)-\psi'(2k+1)}{\G(2k+1)}\left(\frac{4\pi^2mn}{y}\right)^{2k}\Bigg\}.
\end{align*}
Therefore, invoking Lemma \ref{sumpsi}, using \eqref{relation} and Theorem \ref{dzhrbashyan} with $w=4\pi^2mn/y$, we are led to
\begin{align*}
G_3(y)&=\frac{2}{y}\sum_{m, n=1}^{\infty}\hspace{-4pt}\bigg\{\hspace{-3pt}\log\left(\frac{4\pi^2mn}{y}\right)\log\left(\frac{ny}{m}\right)\cosh\left(\frac{4\pi^2mn}{y}\right)-2\log(2n\pi)\bigg[\left(\sinh\operatorname{Shi}-\cosh\operatorname{Chi}\right)\left(\frac{4\pi^2mn}{y}\right)\nonumber\\
&\quad+\log\left(\frac{4\pi^2mn}{y}\right)\cosh\left(\frac{4\pi^2mn}{y}\right)\hspace{-4pt}\bigg]+\log^{2}\left(\frac{4\pi^2mn}{y}\right)\cosh\left(\frac{4\pi^2mn}{y}\right)+2\hspace{-3pt}\int_{0}^{\infty}\frac{u\cos(u)\log(u)\, du}{u^2+(4\pi^2mn/y)^2}\bigg\}\nonumber\\
&=\frac{4}{y}\sum_{n=1}^{\infty}\sum_{m=1}^{\infty}\int_{0}^{\infty}\frac{u\cos(u)\log(u/(2\pi n))}{u^2+(4\pi^2mn/y)^2}\, du,
\end{align*}
where in the last step, we used Lemma \ref{minusonelemma} with $w=4\pi^2mn/y$.

Next, an application of Theorem \ref{analoguedgkm} with $w=2\pi n/y$ and then of \eqref{dgkmresult} with $w=4\pi^2n/y$ in the second step further yields
\begin{align}\label{g3eval}
G_3(y)&=\frac{4}{y}\sum_{n=1}^{\infty}\left\{\sum_{m=1}^{\infty}\int_{0}^{\infty}\frac{u\cos(u)\log(uy/(2\pi n))}{u^2+(4\pi^2mn/y)^2}\, du-\log(y)\sum_{m=1}^{\infty}\int_{0}^{\infty}\frac{u\cos(u)}{u^2+(4\pi^2mn/y)^2}\, du\right\}\nonumber\\
&=\frac{1}{y}\sum_{n=1}^{\infty}\Bigg\{\psi_1\left(\frac{2\pi in}{y}\right)-\frac{1}{2}\log^{2}\left(\frac{2\pi in}{y}\right)+\psi_1\left(-\frac{2\pi in}{y}\right)-\frac{1}{2}\log^{2}\left(-\frac{2\pi in}{y}\right)+\frac{y}{4n}\nonumber\\
&\quad+\left(\gamma+\log(y)\right)\left(\psi\left(\frac{2\pi in}{y}\right)+\psi\left(-\frac{2\pi in}{y}\right)-2\log\left(\frac{2\pi n}{y}\right)\right)\Bigg\}\nonumber\\
&=\frac{1}{y}\sum_{n=1}^{\infty}\Bigg\{\psi_1\left(\frac{2\pi in}{y}\right)-\frac{1}{2}\log^{2}\left(\frac{2\pi in}{y}\right)+\psi_1\left(-\frac{2\pi in}{y}\right)-\frac{1}{2}\log^{2}\left(-\frac{2\pi in}{y}\right)+\frac{y}{4n}\Bigg\}\nonumber\\
&\quad+\frac{\left(\gamma+\log(y)\right)}{y}\sum_{n=1}^{\infty}\Bigg\{\left(\psi\left(\frac{2\pi in}{y}\right)+\psi\left(-\frac{2\pi in}{y}\right)-2\log\left(\frac{2\pi n}{y}\right)\right)\Bigg\},
\end{align}
where the validity of the last step results from the fact that invoking Theorem \ref{asymptotic-psi1}, once with $z=2\pi in/y$, and again with $z=-2\pi in/y$, and adding the two expansions yields
\begin{align}\label{asyseries}
\psi_1\left(\frac{2\pi in}{y}\right)-\frac{1}{2}\log^{2}\left(\frac{2\pi in}{y}\right)+\psi_1\left(-\frac{2\pi in}{y}\right)-\frac{1}{2}\log^{2}\left(-\frac{2\pi in}{y}\right)+\frac{y}{4n}=O_y\left(\frac{\log(n)}{n^2}\right).
\end{align}

Therefore, from \eqref{diffid}, \eqref{g1eval}, \eqref{g2eval} and \eqref{g3eval}, we arrive at
\begin{align*}
&\sum_{n=1}^{\infty}\frac{\log(n)}{e^{ny}-1}+\frac{1}{4}\log(2\pi)-\frac{1}{2y}\log^{2}(y)+\frac{\gamma^2}{2y}-\frac{\pi^2}{12y}\nonumber\\
&=\frac{1}{y}\sum_{n=1}^{\infty}\Bigg\{\psi_1\left(\frac{2\pi in}{y}\right)-\frac{1}{2}\log^{2}\left(\frac{2\pi in}{y}\right)+\psi_1\left(-\frac{2\pi in}{y}\right)-\frac{1}{2}\log^{2}\left(-\frac{2\pi in}{y}\right)+\frac{y}{4n}\Bigg\}\nonumber\\
&\quad+\frac{\left(\gamma+\log(y)\right)}{y}\sum_{n=1}^{\infty}\Bigg\{\left(\psi\left(\frac{2\pi in}{y}\right)+\psi\left(-\frac{2\pi in}{y}\right)-2\log\left(\frac{2\pi n}{y}\right)\right)\Bigg\},
\end{align*}
which, upon rearrangement, gives \eqref{translog}.

Now \eqref{translog1} can be immediately derived from \eqref{translog} by invoking \eqref{kanot} to transform the first series on the right-hand side of \eqref{translog}, multiplying both sides of the resulting identity by $y$ followed by simplification and rearrangement.
\qed

The theorem just proved allows us to obtain the complete asymptotic expansion of $\sum\limits_{n=1}^{\infty}d(n)\log(n)e^{-ny}$ as $y\to0$. This is derived next.

\subsection{Proof of Theorem \ref{logy0} }
From \eqref{asypsi}, as $y\to 0$,
\begin{equation*}
\psi\left(\frac{2\pi i n}{y} \right) \hspace{-2pt}\sim \log\left(\frac{2\pi i n}{y} \right) - \sum_{k=1}^\infty \frac{B_{2k}}{2k}\left(\frac{2\pi i n}{y} \right)^{-2k}
\hspace{-6pt} \text{and}  \hspace{.2cm}
\psi\left(-\frac{2\pi i n}{y} \right) \hspace{-2pt}\sim \log\left(-\frac{2\pi i n}{y} \right) - \sum_{k=1}^\infty \frac{B_{2k}}{2k}\left(-\frac{2\pi i n}{y} \right)^{-2k}\hspace{-5pt}.
\end{equation*}
Inserting the above asymptotics into the first sum of the right-hand side of \eqref{translog}, we obtain
\begin{align}\label{First sum}
\sum_{n=1}^\infty \left\{\log\left(\frac{2\pi n}{y} \right) -\frac{1}{2}\left(\psi\left(\frac{2\pi i n}{y} \right) + \psi\left(-\frac{2\pi i n}{y} \right) \right) \right\} &\sim \sum_{n=1}^\infty \sum_{k=1}^\infty \frac{(-1)^k B_{2k}}{2k}\left(\frac{2\pi n}{y} \right)^{-2k}\nonumber\\
&= \sum_{k=1}^\infty \frac{(-1)^k B_{2k}\zeta(2k)}{2k}\left(\frac{2\pi}{y} \right)^{-2k} \nonumber\\
&=-\frac{1}{2} \sum_{k=1}^\infty \frac{B_{2k}^2 y^{2k}}{(2k)(2k)!},
\end{align}
where in the last step we used Euler's formula 
\begin{equation}\label{eulerf}
\zeta(2 m ) =
	(-1)^{m +1} \displaystyle\frac{(2\pi)^{2 m}B_{2 m }}{2 (2 m)!}\hspace{5mm} (m\geq0).
	\end{equation}
	 Next, we need to find the asymptotic expansion of the second sum on the right-hand side of \eqref{translog}. 
Invoking Theorem \ref{asymptotic-psi1} twice, once with $z=2\pi in/y$, and again, with $z=-2\pi in/y$,  and inserting them into the second sum of the right-hand side of \eqref{translog}, we obtain
\begin{align}\label{Second sum}
&\sum_{n=1}^{\infty}\left\{\psi_1\left(\frac{2\pi in}{y}\right)+\psi_1\left(-\frac{2\pi in}{y}\right)-\frac{1}{2}\left(\log^{2}\left(\frac{2\pi in}{y}\right)+\log^{2}\left(-\frac{2\pi in}{y}\right)\right)+\frac{y}{4n}\right\}\nonumber\\
& \sim \sum_{n=1}^{\infty} \sum_{k=1}^{\infty} \frac{(-1)^k B_{2k}}{k}\left(\frac{y}{2\pi n}\right)^{2k}\left[\frac{1}{2k-1} + \sum_{j=1}^{2k-2}\frac{1}{j}-  \log \left(\frac{2\pi n}{y} \right) \right]\nonumber\\
&= \sum_{n=1}^{\infty} \sum_{k=1}^{\infty} \frac{(-1)^k B_{2k}}{k}\left(\frac{y}{2\pi n}\right)^{2k}\left[\frac{1}{2k-1} + \left(\sum_{j=1}^{2k-2}\frac{1}{j}-\log \left(\frac{2\pi}{y} \right)\right) -\log n \right]\nonumber\\
&= \sum_{k=1}^{\infty} \frac{(-1)^k B_{2k}}{k}\left(\frac{y}{2\pi}\right)^{2k}\left[\zeta(2k)\left\{\frac{1}{2k-1} + \left(\sum_{j=1}^{2k-2}\frac{1}{j}-\log \left(\frac{2\pi}{y} \right)\right)\right\}+\zeta'(2k) \right] \nonumber\\
&= -\sum_{k=1}^\infty\left[ \frac{B_{2k}^2 y^{2k}}{(2k-1)(2k)(2k)!} + \frac{B_{2k}^2 y^{2k}}{2k(2k)!}\left(\sum_{j=1}^{2k-2}\frac{1}{j}-\log \left(\frac{2\pi}{y} \right)\right) +  \frac{(-1)^{k+1} B_{2k}\zeta'(2k)}{k}\left(\frac{y}{2\pi}\right)^{2k}\right],
\end{align}
where in the penultimate step we used $\zeta'(s)=-\sum_{n=1}^{\infty}\log(n)n^{-s}$ for Re$(s)>1$, and in the last step we used \eqref{eulerf}. Substituting \eqref{First sum} and \eqref{Second sum} in \eqref{translog}, we see that as $y\to 0$,
\begin{align*}
&\sum_{n=1}^{\infty}\frac{\log(n) }{e^{ny}-1} \nonumber\\
&\sim -\frac{1}{4}\log(2\pi)+\frac{1}{2y}\log^{2}(y)-\frac{\gamma^2}{2y}+\frac{\pi^2}{12y}
-\frac{2}{y}(\gamma+\log(y))\left\{-\frac{1}{2} \sum_{k=1}^\infty \frac{B_{2k}^2 y^{2k}}{(2k)(2k)!}\right\}\nonumber\\
&\quad -\frac{1}{y}\sum_{k=1}^\infty\left[ \frac{B_{2k}^2 y^{2k}}{(2k-1)(2k)(2k)!} + \frac{B_{2k}^2 y^{2k}}{2k(2k)!}\left(\sum_{j=1}^{2k-2}\frac{1}{j}-\log \left(\frac{2\pi}{y} \right)\right) +  \frac{(-1)^{k+1} B_{2k}\zeta'(2k)}{k}\left(\frac{y}{2\pi}\right)^{2k}\right]\nonumber\\
& = -\frac{1}{4}\log(2\pi)+\frac{1}{2y}\log^{2}(y)-\frac{\gamma^2}{2y}+\frac{\pi^2}{12y} +\hspace{-2pt} \sum_{k=1}^{\infty} \frac{B_{2k} y^{2k-1}}{k}\hspace{-2pt}\left\{\frac{B_{2k}}{2(2k)!}\left(\gamma -\hspace{-2pt} \sum_{j=1}^{2k-1}\frac{1}{j}+\log (2\pi)\right)\hspace{-4pt}+\hspace{-2pt} \frac{(-1)^k \zeta'(2k)}{(2\pi)^{2k}}\right\}\nonumber\\
&= -\frac{1}{4}\log(2\pi)+\frac{1}{2y}\log^{2}(y)-\frac{\gamma^2}{2y}+\frac{\pi^2}{12y} + \frac{y}{6}\bigg\{\frac{1}{24}(\gamma-1+\log(2\pi)) -\frac{1}{4\pi^2} \bigg[\frac{\pi^2}{6}(\gamma + \log(2\pi) \nonumber\\
&\quad- 12\log (A)) \bigg] \bigg\}
+\sum_{k=2}^{\infty} \frac{B_{2k} y^{2k-1}}{k}\left\{\frac{B_{2k}}{2(2k)!}\left(\gamma - \sum_{j=1}^{2k-1}\frac{1}{j}+\log (2\pi)\right)+ \frac{(-1)^k \zeta'(2k)}{(2\pi)^{2k}}\right\}\nonumber\\
&= \frac{1}{2y}\log^{2}(y) + \frac{1}{y}\left(\frac{\pi^2}{12} - \frac{\gamma^2}{2} \right)-\frac{1}{4}\log(2\pi) +\frac{y}{12}\left(\log (A) - \frac{1}{12} \right) \nonumber\\
&\quad +\sum_{k=2}^{\infty} \frac{B_{2k} y^{2k-1}}{k}\left\{\frac{B_{2k}}{2(2k)!}\left(\gamma - \sum_{j=1}^{2k-1}\frac{1}{j}+\log (2\pi)\right)+ \frac{(-1)^k \zeta'(2k)}{(2\pi)^{2k}}\right\}.
\end{align*}
This completes the proof.
\qed

An application of the result proved above in the theory of the moments of the Riemann zeta function is given next.
\subsection{Proof of Theorem \ref{moments}}
	Our goal is to link $\int\limits_{0}^{\infty}\zeta\left(\frac{1}{2}-it\right)\zeta'\left(\frac{1}{2}+it\right)e^{-\delta t}\, dt$ with $\sum\limits_{n=1}^{\infty}\frac{\log(n) }{\exp\left(2 \pi i n e^{-i\d}\right)-1}$ and then invoke Theorem \ref{logy0}. Our treatment is similar to that of Atkinson \cite{atkinson}.
	
First note that, for $\textup{Re}(y)>0$,
\begin{align}\label{before_interchange}
	\sum_{n=1}^{\infty} d(n)\log(n) e^{-ny}=\sum_{n=1}^{\infty} d(n)\log( n) {1\over 2\pi i}\int\limits_{(c)}\Gamma(s)(ny)^{-s}ds,
\end{align}
where $c=\textup{Re}(s)>1$. Using $(1*\log)(n)=\sum_{d|n}\log d=\frac{1}{2}d(n)\log( n)$ and interchanging the order of summation and integration on the right-hand side of \eqref{before_interchange}, we obtain
\begin{align}\label{base1}
	\sum_{n=1}^{\infty} d(n)\log( n) e^{-ny}&={1\over 2\pi i}\int\limits_{(c)}\Gamma(s)\left( \sum_{n=1}^{\infty} {d(n)\log (n)\over n^s }\right) y^{-s}ds\nonumber\\  
	&={2\over 2\pi i}\int\limits_{(c)}\Gamma(s) \sum_{n=1}^{\infty} {(1*\log) (n)\over n^s } y^{-s}ds\nonumber\\  
	&=-{2\over 2\pi i}\int\limits_{(c)}\Gamma(s) \z(s)\z'(s)  y^{-s}ds.
\end{align} 
 Letting $y=2 \pi i e^{-i\d}$ so that $\textup{Re}(\d)>0$ , we deduce that 
 \begin{align}\label{seriesdnlogn}
 	\sum_{n=1}^{\infty} d(n)\log( n) \exp\left(-2 \pi i n e^{-i\d}\right)&=-{2\over 2\pi i}\int\limits_{(c)}\Gamma(s) \z(s)\z'(s) (2 \pi i e^{-i\d})^{-s}ds.
 \end{align}
We would like to shift the line of integration from Re$(s)=c$ to Re$(s)=1/2$. The integrand in \eqref{seriesdnlogn} has a third order pole at $s=1$. Therefore, using Cauchy's residue theorem on the rectangular contour with sides $\left[c -iT, c+iT \right],  \left[c+iT, {1\over 2} + iT \right], \left[{1\over 2} + iT, {1\over 2}-iT \right]$ and $ \left[{1\over 2}-iT, c-iT \right] $, where $T>0$, noting that \eqref{strivert} implies that the integrals along the horizontal line segments tend to zero as $T \to \infty$, we obtain
\begin{align}\label{Cauchythmdnlogn}
	&{1\over 2\pi i}\int\limits_{(c)}\Gamma(s) \z(s)\z'(s) (2 \pi i e^{-i\d})^{-s}ds\nonumber\\
	&={1\over 2\pi i}\int\limits_{({1\over 2})}\Gamma(s) \z(s)\z'(s) (2 \pi i e^{-i\d})^{-s}ds+{\Res_{s=1}}\Gamma(s) \z(s)\z'(s) (2 \pi i e^{-i\d})^{-s}\nonumber \\
	&= {1\over 4\pi i}\int\limits_{({1\over 2})}{\zeta(1-s) \z'(s) ( i e^{-i\d})^{-s}\over \cos\left({\pi s\over 2} \right) }ds+{\frac{1}{2 \pi i e^{-i\d}}}\left({\gamma^2 \over 2 }-{\pi^2\over 12}-\frac{1}{2}\log ^2(2 \pi i e^{-i\d}) \right) ,
\end{align}
where, in the last step, we used the functional equation \eqref{zetaasym} with $s$ replaced by $1-s$. 
Hence, from \eqref{seriesdnlogn} and \eqref{Cauchythmdnlogn}, we arrive at 
\begin{align}\label{equation zeta1}
	{e^{-{i\d\over 2}}\over 2 i}\int\limits_{(\frac{1}{2})}{\zeta(1-s) \z'(s) e^{-is\left({\pi\over 2}-\d \right) }\over \cos\left({\pi s\over 2} \right) }ds&=-\pi e^{-{i\d\over 2}}\sum_{n=1}^{\infty} d(n)\log( n) \exp\left(-2 \pi i n e^{-i\d}\right)\nonumber\\
	&\quad+{ie^{{i\d\over 2}}\over 2}\left({\gamma^2 }-{\pi^2\over 6}-{\log ^2(2 \pi i e^{-i\d})} \right).
\end{align}
We next consider the difference of the integrals on the left-hand sides of  \eqref{mvtder} and \eqref{equation zeta1} and employ the change of variable $s={1\over 2}+it$ in the second integral in the first step below to see that
\begin{align}
	\int_{0}^{\infty}&\z\left({1\over 2}-it\right)\z'\left({1\over 2}+it\right)e^{-\d t}dt-{e^{-{i\d\over 2}}\over i}\int_{\left( 1\over 2\right) }\frac{\z\left({1-s}\right)\z'\left(s\right)}{2\cos\left(\pi s \over 2 \right) }e^{-is \left({\pi\over 2}-\d \right)}ds\nonumber \\
	&=\int_{0}^{\infty}\z\left({1\over 2}-it\right)\z'\left({1\over 2}+it\right)e^{-\d t}dt-{e^{-{i\d\over 2}}}\int_{-\infty}^{\infty}\frac{\z\left({1\over 2}-it\right)\z'\left({1\over 2}+it\right)}{\left(e^{{i\pi  \over 2 }\left({1\over 2}+it\right)} +e^{-{i\pi  \over 2 }\left({1\over 2}+it\right)}\right) }e^{-i\left({1\over 2}+it \right) \left({\pi\over 2}-\d \right)}dt\nonumber \\
	&=\int_{0}^{\infty}\z\left({1\over 2}-it\right)\z'\left({1\over 2}+it\right)e^{-\d t}dt+\int_{0}^{\infty}\frac{\z\left({1\over 2}+it\right)\z'\left({1\over 2}-it\right)}{\left( e^{{i\pi  \over 4 }+{\pi t\over 2}}  +e^{-{i\pi  \over 4 }-{\pi t\over 2}} \right)  }e^{-{i\pi  \over 4 }-{\pi t\over 2}+\d t}dt\nonumber \\
	& \qquad\qquad\qquad\qquad\qquad\qquad\qquad\qquad
	-\int_{0}^{\infty}\frac{\z\left({1\over 2}-it\right)\z'\left({1\over 2}+it\right)}{\left( e^{{i\pi  \over 4 }-{\pi t\over 2}}  +e^{-{i\pi  \over 4 }+{\pi t\over 2}} \right)  }e^{-{i\pi  \over 4 }+{\pi t\over 2}-\d t}dt\nonumber \\
	&=\int_{0}^{\infty}\frac{\z\left({1\over 2}-it\right)\z'\left({1\over 2}+it\right)}{\left( e^{{i\pi  \over 4 }-{\pi t\over 2}}  +e^{-{i\pi  \over 4 }+{\pi t\over 2}} \right)  }e^{{i\pi  \over 4 }-{\pi t\over 2}-\d t}dt
	+\int_{0}^{\infty}\frac{\z\left({1\over 2}+it\right)\z'\left({1\over 2}-it\right)}{\left( e^{{i\pi  \over 4 }+{\pi t\over 2}}  +e^{-{i\pi  \over 4 }-{\pi t\over 2}} \right)  }e^{-{i\pi  \over 4 }-{\pi t\over 2}+\d t}dt.\label{Integrals split}
\end{align}

From \cite[p. 127, Eq. 20]{gonek}, we have
\begin{align*}
	\z\left({1\over 2}-it\right)\ll |t|^{{1\over 4}+{\e \over 2}}\hspace{10pt}\text{and} \hspace{10pt}	\z'\left({1\over 2}+it\right)\ll |t|^{{1\over 4}+{\e \over 2}},
\end{align*}
 whence, by reverse triangle inequality, we have 
 \begin{align*}
 	\left|\frac{\z\left({1\over 2}-it\right)\z'\left({1\over 2}+it\right)}{\left( e^{{i\pi  \over 4 }-{\pi t\over 2}}  +e^{-{i\pi  \over 4 }+{\pi t\over 2}}  \right) }e^{{i\pi  \over 4 }-{\pi t\over 2}-\d t}\right|\ll\frac{|t|^{{1\over 2}+\epsilon}e^{-\textup{Re}(\d) t-{\pi t \over 2}}}{e^{{\pi t \over 2}}-e^{-{\pi t \over 2}}}=\frac{|t|^{{1\over 2}+\epsilon}e^{(-\textup{Re}(\d)-\pi) t}}{1-e^{{-\pi t }}}.
 \end{align*}
This implies that the first integral is analytic for $\textup{Re}(\d) >-\pi $. Similarly, the second integral is analytic for $\textup{Re}(\d) <\pi$. Consequently, both integrals of \eqref{Integrals split} are analytic in $\d$ in $|\d|<\pi$.
 
 From \eqref{equation zeta1} and above discussion, we can say that the expression  
 \begin{align}\label{defphi}
 	\phi (\d)&:=\int_{0}^{\infty}\z\left({1\over 2}-it\right)\z'\left({1\over 2}+it\right)e^{-\d t}dt+\pi e^{-{i\d\over 2}}\sum_{n=1}^{\infty} d(n)\log( n) \exp\left(-2 \pi i n e^{-i\d}\right)\nonumber\\
 	&\quad-ie^{{i\d\over 2}}\left({\gamma^2 \over 2 }-{\pi^2\over 12}-{\log ^2(2 \pi i e^{-i\d})\over 2} \right)
 \end{align}
is an analytic function of $\d$ in $|\d|<\pi$ except at $\d=-{\pi \over 2}$, because the third term of \eqref{defphi} is not analytic at $\d=-{\pi \over2}$. Hence, $\phi(\d)+ie^{{i\d\over 2}}\left({\gamma^2 \over 2 }-{\pi^2\over 12}-{\log ^2(2 \pi i e^{-i\d})\over 2} \right)$ can be expressed as a power series in $\d$ for $\{\delta:|\d |<\pi\}\backslash\{-{\pi \over 2}\}$. Letting $y=2\pi i  (e^{-i\d}-1)$ in the asymptotic expansion of $\sum_{n=1}^{\infty} d(n)\log (n) e^{-ny}$ in \eqref{logy0eqn}, we find that as $\d \to 0$,
\begin{align}\label{estdnlogn}
&\pi e^{-{i\d\over 2}}\sum_{n=1}^{\infty} d(n)\log( n) \exp\left(-2 \pi i n e^{-i\d}\right)\nonumber \\
&=\pi e^{-{i\d\over 2}}\sum_{n=1}^{\infty} d(n)\log( n) \exp\left(-2 \pi i n (e^{-i\d}-1)\right)\nonumber\\
&= \pi e^{-{i\d\over 2}}\left[ {1\over 2 \pi i (e^{-i\d}-1)}\left(\log ^2(2 \pi i (e^{-i\d}-1))+ {\pi^2\over 6}-{\gamma^2  }\right) -\frac{1}{2}\log(2\pi)\right. \nonumber \\
&\quad+\sum_{k=2}^{m-1}\frac{B_{2k}(2\pi i (e^{-i\d}-1))^{2k-1}}{k}\left\lbrace {B_{2k}\over (2k)!}\left(\g-\sum_{j=1}^{2k-1}{1\over j}+\log(2\pi) \right)+\frac{2(-1)^k\z'(2k)}{(2\pi)^{2k}}  \right\rbrace  \nonumber \\
&\quad+\left. {2 \pi i (e^{-i\d}-1)\over 6}\left(\log A -{1\over 12}\right) +O\left( (e^{-i\d}-1)^{2m-1}\right)\right]  \nonumber \\
&={1\over 4\sin\left(\d\over 2\right)}\left(\log ^2(2 \pi \d)+ {\pi^2\over 6}-{\gamma^2  }\right)-\frac{\pi }{2} e^{-{i\d\over 2}}\log (2\pi)\nonumber \\
&\quad +\sum_{k=1}^{m-1}\pi c_ke^{-{i\d\over 2}}( 2 \pi i (e^{-i\d}-1))^{2k-1}+O\left( (e^{-i\d}-1)^{2m-1}\right) \nonumber \\
&={1\over 4\sin\left(\d\over 2\right)}\left( {\pi^2\over 6}-{\gamma^2  }+\log ^2(2 \pi \d)\right)  +\sum_{k=0}^{2m-2}d_k\d^k +O(\d^{2m-1}),
\end{align} 
as $\d\to 0$, where $c_k$ and $d_k$ are effectively computable constants. 
Consequently, from \eqref{defphi} and \eqref{estdnlogn}, as $\d\to 0$, we get
\begin{align*}
	\int_{0}^{\infty}&\z\left({1\over 2}-it\right)\z'\left({1\over 2}+it\right)e^{-\d t}dt={1\over 4\sin\left(\d\over 2\right)}\left( {\gamma^2}-{\pi^2\over 6}-\log ^2(2 \pi \d)\right)  +\sum_{k=0}^{2m-2}d_k\d^k +O(\d^{2m-1}).
\end{align*}
This completes the proof.
\qed

\section{Concluding remarks}
The main highlight of this paper was to derive \eqref{translog}, that is, an exact transformation for the series $\displaystyle\sum_{n=1}^{\infty}\frac{\log(n)}{e^{ny}-1}$, or equivalently, for $\sum_{n=1}^{\infty}d(n)\log(n)e^{-ny}$, where $\textup{Re}(y)>0$. Such a transformation was missing from the literature up to now. One of the reasons for this could be the lack of availability of the transformation in \eqref{maineqn} until \cite{dkk} appeared, which, in turn, resulted from evaluating an integral with a combination of Bessel functions as its associated kernel. This underscores the importance of the applicability of integral transforms in number theory.

 In addition, it is to be emphasized that in order to derive \eqref{translog}, several new intermediate results, interesting in themselves, had to be derived, for example, the ones given in Theorems \ref{asymptotic-psi1},  \ref{kloosterman-type}, \ref{analoguedgkm}, and \ref{dzhrbashyan}. This also shows how important $\psi_1(z)$, and, in general, $\psi_k(z)$, are in number theory, and further corroborates Ishibashi's quote given in the introduction.

The transformation in  \eqref{translog} has a nice application in the study of moments of $\zeta(s)$, namely, to obtain the asymptotic expansion of 
$ \int_{0}^{\infty}\zeta\left(\frac{1}{2}-it\right)\zeta'\left(\frac{1}{2}+it\right)e^{-\delta t}\, dt$ as $\delta\to0$.  However, considering the fact that \eqref{translog} holds for \emph{any} $y$ with Re$(y)>0$, we expect further applications of it in the future, in particular, in the theory of modular forms.

A generalization of \eqref{translog}, equivalently, of \eqref{translog1},  may be obtained for the series $\sum_{n=1}^{\infty}d(n)\log^{k}(n)e^{-ny}$. This would require differentiating \eqref{maineqn} $k$ times with respect to $a$ and then letting $a\to0$. This would involve Ishibashi's higher analogues of Deninger's $R(z)$ defined in \eqref{rkz}. Ishibashi \cite[Theorem 1]{ishibashi} gave a Plana-type formula for $R_k(x), x>0$, which is easily seen to hold for complex $x$ in the right-half plane Re$(x)>0$. This formula involves a polynomial in $\log(t)$ defined by
	$S_k(t):=\sum_{j=0}^{k-1}a_{k, j}\log^{j}(t)$, 
where $a_{k, j}$ are recursively defined by
	$a_{k, j}=-\sum_{r=0}^{k-2}\binom{k-1}{r}\G^{(k-r-1)}(1)a_{r+1, j}, 0\leq j\leq k-2$,
with $a_{1,0}=1$ and $a_{k, k-1}=1$. 

Observe that $S_1(t)=1$ and $S_2(t)=\g+\log(t)$, and so the numerators of the summands of the series on the left-hand sides of \eqref{kanot} and \eqref{translog1} are $S_1(ny)$ and $S_2(ny)$ respectively. In view of this, we speculate the left-hand side of \eqref{translog1} to generalize to
$\sum\limits_{n=1}^{\infty}\displaystyle\frac{S_k(ny)}{e^{ny}-1}$ and the right-hand side to involve $\psi_k(z)$.

Since \eqref{base1} can be generalized to
\begin{align*}
	\sum_{n=1}^{\infty}\frac{\log^{k}(n)}{e^{ny}-1}=-\frac{1}{2\pi i}\int_{(c)}\Gamma(s)\zeta(s)\zeta^{(k)}(s)y^{-s}\, ds\hspace{5mm}(c=\textup{Re}(s)>1, \textup{Re}(y)>0),
	\end{align*}
on account of the fact that $\zeta^{(k)}(s)=(-1)^k\sum\limits_{n=1}^{\infty}\log^{k}(n)n^{-s}$, once a generalization of \eqref{translog1} is obtained, it would possibly give us the complete asymptotic expansion of $ \int_{0}^{\infty}\zeta\left(\frac{1}{2}-it\right)\zeta^{(k)}\left(\frac{1}{2}+it\right)e^{-\delta t}\, dt$ as $\delta\to0$.\\

\begin{center}
	\textbf{Acknowledgements}
	\end{center}
	The authors sincerely thank Christopher Deninger, Nico M. Temme, Rahul Kumar, Shigeru Kanemitsu, Steven M. Gonek, Alessandro Languasco and Caroline Turnage-Butterbaugh for helpful discussions. They also thank Professor Inese Bula from University of Latvia for sending a copy of \cite{riekstins}. The first author is a National Postdoctoral Fellow at IIT Gandhinagar funded by the grant. The second author's research is funded by the Swarnajayanti Fellowship grant SB/SJF/2021-22/08. Both sincerely thank SERB for its support. The third author acknowledges the support of CSIR SPM Fellowship under the grant number SPM-06/1031(0281)/2018-EMR-I.

	\end{document}